\documentclass{amsart}
\usepackage[T1]{fontenc}
\usepackage[utf8]{inputenc}
\usepackage{lmodern}
\usepackage{microtype}
\usepackage{amsmath,amssymb,amsthm,mathtools}
\usepackage{thmtools}
\usepackage{thm-restate}
\usepackage{bm}
\usepackage{tikz}
\usetikzlibrary{cd}
\usepackage{multicol}
\usepackage{array}
\usepackage[shortlabels]{enumitem}

\usepackage[hyphens]{url}
\usepackage{nicefrac}

\usepackage[colorlinks]{hyperref}
\hypersetup{
    colorlinks,
    linkcolor={red!50!black},
    citecolor={blue!50!black},
    urlcolor={blue!80!black}
}

\newcommand\seq[1]{\langle #1 \rangle}

\newcommand\ceil[1]{\lceil #1 \rceil}

\newcommand{\conv}{\!\downarrow\;}

\DeclareMathOperator{\powerset}{\mathcal{P}}
\DeclareMathOperator{\dom}{dom}

\DeclareMathOperator{\uh}{\upharpoonright}
\DeclareMathOperator{\upperd}{\bar{\rho}}
\DeclareMathOperator{\lowerd}{\underline{\rho}}

\def\N{\mathbb{N}}

\def\Q{\mathbb{Q}}

\def\X{\mathcal{X}}

\def\F{\mathcal{F}}
\def\CC{\mathcal{C}}

\def\D{\mathcal{D}}
\def\Cantor{2^\omega}

\def\RT{\mathsf{RT}}
\def\DS{\mathsf{SD}}
\def\ACA{\mathsf{ACA}}
\def\RCA{\mathsf{RCA}}
\def\WKL{\mathsf{WKL}}

\DeclareMathOperator{\Var}{Var}
\DeclareMathOperator{\E}{E}

\let\0\varnothing
\let\bar\overline
\let\tilde\widetilde
\let\phi\varphi
\let\term\emph

\renewcommand{\epsilon}{\varepsilon}

\newcommand{\concat}{^\smallfrown}

\theoremstyle{plain}
\newtheorem{theorem}{Theorem}[section]

\newtheorem*{claim*}{Claim}
\newtheorem{proposition}[theorem]{Proposition}
\newtheorem*{proposition*}{Proposition}

\newtheorem*{fact*}{Fact}
\newtheorem{conjecture}[theorem]{Conjecture}
\newtheorem*{conjecture*}{Conjecture}
\newtheorem{corollary}[theorem]{Corollary}
\newtheorem{lemma}[theorem]{Lemma}
\newtheorem*{lemma*}{Lemma}

\newtheorem{question}[theorem]{Question}
\newtheorem*{question*}{Question}
\theoremstyle{definition}
\theoremstyle{definition}\newtheorem*{remark*}{Remark}
\theoremstyle{definition}\newtheorem{definition}[theorem]{Definition}
\theoremstyle{definition}
\theoremstyle{definition}
\theoremstyle{definition}\newtheorem*{example*}{Example}
\theoremstyle{definition}
\theoremstyle{definition}

\title{Coding information into all infinite subsets of a dense set}
\author{Matthew Harrison-Trainor}
\thanks{This work was supported by the National Science Foundation under Grant 	DMS-2153823 and Grant DMS-2203072.}
\author{Lu Liu}
\author{Patrick Lutz}

\begin{document}

\maketitle
\thispagestyle{empty}
\vspace{-20pt}
\begin{abstract}
Suppose you have an uncomputable set $X$ and you want to find a set $A$, all of whose infinite subsets compute $X$. There are several ways to do this, but all of them seem to produce a set $A$ which is fairly sparse. We show that this is necessary in the following technical sense: if $X$ is uncomputable and $A$ is a set of positive lower density then $A$ has an infinite subset which does not compute $X$. We also prove an analogous result for PA degree: if $X$ is uncomputable and $A$ is a set of positive lower density then $A$ has an infinite subset which is not of PA degree. We will show that these theorems are sharp in certain senses and also prove a quantitative version formulated in terms of Kolmogorov complexity. Our results use a modified version of Mathias forcing and build on work by Seetapun, Liu, and others on the reverse math of Ramsey's theorem for pairs.
\end{abstract}

\section{Introduction}

Suppose you have an uncomputable set $X$ and would like to find an infinite set $A \subseteq \N$ such that all infinite subsets of $A$ compute $X$. Here's one way to do this, due to Dekker and Myhill~\cite{dekker1958retraceable}: identify $\N$ with $2^{<\omega}$ and let $A$ be the set of all finite initial segments of $X$. 

This is not the only way to encode an uncomputable set $X$ into all infinite subsets of $A$. For example, if $X$ is hyperarithmetic then it can be computed from any sufficiently fast growing function (\cite{jockusch1969uniformly}, Theorem 6.8). If we make sure $A$ has sufficiently large gaps between its elements then for any infinite subset $B$ of $A$, the function which enumerates the elements of $B$ grows fast enough to compute $X$ and hence $B$ itself computes $X$.

Note that both of these methods produce fairly sparse subsets of $\N$. If we use Dekker and Myhill's method then the set $A$ will have just $n$ elements less than $2^n$. If we use the second method (in the case where $X$ is hyperarithmetic) then the set $A$ will be even sparser---the gaps between successive elements of $A$ grow faster than any computable function.

It seems reasonable to informally conjecture that this is a necessary feature of such coding methods; in other words, to conjecture that if $X$ is uncomputable and every infinite subset of $A$ computes $X$ then $A$ must be sparse. We can turn this informal conjecture into a formal one by picking a precise definition of ``sparse.''

The main theorem of this paper states that the conjecture holds if we define ``sparse'' to mean ``lower density zero.'' Recall that the \term{lower density} of a set of natural numbers $A \subseteq \N$ is
\[
  \lowerd(A) = \liminf_{n \to \infty} \frac{|A \cap [n]|}{n + 1}
\]
where $[n]$ denotes the set $\{0, 1, \ldots, n\}$. The main theorem of this paper is as follows.

\begin{restatable}{theorem}{main}
\label{thm:main}
For any uncomputable set $X$ and any set $A \subseteq \N$ of positive lower density, there is some infinite subset of $A$ which does not compute $X$.
\end{restatable}

Our proof of this theorem relies on a theorem implicit in the work of Seetapun~\cite{seetapun1995strength} and first proved explicitly by Dzhafarov and Jockusch~\cite{dzhafarov2009ramsey}.\footnote{This theorem is sometimes known by the name ``strong cone avoidance for $\RT^1_2$,'' which originates from its connection to the reverse math of Ramsey's theorem. We decided not to use that name here since we are not concerned with reverse math in this paper.} For a proof, see~\cite{hirschfeldt2015slicing}, Theorem 6.63.

\begin{restatable}[Seetapun's theorem]{theorem}{seetapun}
\label{thm:seetapun}
For any uncomputable set $X$ and any set $A \subseteq \N$, either $A$ or $\bar{A}$ must contain some infinite subset which does not compute $X$.
\end{restatable}

Our proof is partially inspired by proofs of results related to Seetapun's theorem by Cholak, Jockusch and Slaman~\cite{cholak2001strength}, Dzhafarov and Jockusch~\cite{dzhafarov2009ramsey} and Monin and Patey~\cite{monin2019pigeons}. More specifically, those results are proved using variations on Mathias forcing and use the low basis theorem (or the cone avoiding basis theorem) to show that certain sets of conditions are dense. Our proof also uses a variation on Mathias forcing and the cone avoiding basis theorem, but, in addition, uses Seetapun's theorem in a manner similar to the cone avoiding basis theorem. We will explain our strategy more carefully in Section~\ref{sec:main}.

It is natural to ask whether Theorem~\ref{thm:main} holds for stronger notions of sparsity. In Section~\ref{sec:sharp}, we will show that our theorem is sharp in the sense that it fails to hold for several such notions.

We will also prove two other results which show that it is difficult to encode information into all infinite subsets of a dense set. These two results are stated in terms of having PA degree and in terms of Kolmogorov complexity, respectively.

\subsection*{Avoiding PA degree}

Theorem~\ref{thm:main} can be rephrased in terms of the property of \term{cone avoidance}. Say that a set $\CC \subseteq \Cantor$ \term{avoids cones} if for every nontrivial cone of Turing degrees, there is some element of $\CC$ which is not in that cone. Then Theorem~\ref{thm:main} can be restated as: for every set $A \subseteq \N$ of positive lower density, the set of infinite subsets of $A$ avoids cones.

Cone avoidance can be seen as a kind of computability-theoretic weakness. Another standard weakness notion is \term{PA avoidance}. Say that a set $\CC \subseteq \Cantor$ \term{avoids PA degree} if there is some element of $\CC$ which is not of PA degree. Though cone avoidance and PA avoidance are not equivalent, they do occur together relatively frequently.

One example of this comes from research on $\RT^2_2$, a statement of Ramsey theory which has been thoroughly studied in the field of Reverse Mathematics. Seetapun showed that over $\RCA_0$, $\RT^2_2$ does not imply $\ACA_0$. A core part of the proof is Seetapun's theorem above, which can be read as stating that for any set $A \subseteq \N$, the set of infinite subsets of $A$ and $\bar{A}$ avoids cones. Later, Liu showed that $\RT^2_2$ also does not imply $\WKL_0$ and the heart of his proof was a theorem---analogous to Seetapun's theorem---stating that for every set $A \subseteq \N$, the set of infinite subsets of $A$ and $\bar{A}$ avoids PA degree. 

\begin{theorem}[Liu's Theorem, \cite{liu2012}]\label{thm:liu-thm}
    For any set $A \subseteq \mathbb{N}$, either $A$ or $\overline{A}$ must contain an infinite set that is not of PA degree.
\end{theorem}

It seems natural to ask whether our main theorem can be modified to yield PA avoidance rather than cone avoidance, in the same way that Seetapun's theorem can be modified to yield Liu's theorem. We show that this is indeed the case.

\begin{restatable}{theorem}{mainPA}
\label{thm:mainPA}
   For any set $A \subseteq \mathbb{N}$ of positive lower density, there is an infinite subset of $A$ which is not of PA degree.
\end{restatable}

Just as our proof of Theorem~\ref{thm:main} uses Seetapun's theorem, our proof of the theorem above uses Liu's theorem (along with several techniques first developed to prove Liu's theorem). We will also show that a certain natural-sounding common generalization of both Theorem~\ref{thm:main} and Theorem~\ref{thm:mainPA} is false.

\subsection*{Kolmogorov complexity}

We will also investigate a quantitative version of Theorem~\ref{thm:main}. Stated loosely, that theorem says that it is impossible to encode an infinite amount of information into all infinite subsets of a set of positive lower density. On the other hand, it is obvious that some finite information may be so encoded. For example, it is easy to encode one bit of information into all infinite subsets of a set of lower density $1/2$ (by using the parity of the elements of the set) and, more generally, $n$ bits of information into all infinite subsets of a set of lower density $1/2^n$. But just how much information can be encoded?

We can formulate a precise version of this question using Kolmogorov complexity. To do so, it is convenient to introduce the following definition.

\begin{definition}
For a string $\sigma \in 2^{< \omega}$ and a family of sets $\F \subseteq \powerset(\N)$, define
\[
  C(\sigma \mid \F) = \max_{X \in \F}C^X(\sigma).
\]
\end{definition}

Informally, $C(\sigma \mid \F)$ should be thought of as the Kolmogorov complexity of $\sigma$ relative to an arbitrary element of $\F$. Recall that for any set $A\subseteq \N$, $[A]^\omega$ denotes the family of all infinite subsets of $A$. The quantity
\[
  C(\sigma) - C(\sigma \mid [A]^\omega)
\]
can be thought of as the number of bits of information about $\sigma$ that are encoded into all infinite subsets of $A$. Thus a somewhat more formal version of our question above is: if $\sigma$ is any string and $A \subseteq \N$ is a set of lower density at least $\delta > 0$, then how large can $C(\sigma) - C(\sigma \mid [A]^\omega)$ be?

Based on our example above, it is perhaps natural to guess that this difference should not be much larger than $\log(1/\delta)$---in other words, that it should not be possible to encode more than about $\log(1/\delta)$ bits of information about $\sigma$ into all infinite subsets of a set of lower density $\delta$. Surprisingly, this is not the case.

\begin{restatable}{proposition}{upperbound}\label{prop:ds_kolmogorov}
For any string $\sigma$ and $\delta \in (0, 1]$, there is some set $A\subseteq \N$ of lower density at least $\delta$ such that
\[
  C(\sigma \mid [A]^\omega) \leq \max(0, C^{0'}(\sigma) - \log(1/\delta)) + O(\log\log(1/\delta)).
\]
\end{restatable}

In other words, in addition to lowering the complexity of $\sigma$ by $\log(1/\delta)$, we can also lower it to the $0'$ complexity of $\sigma$. In fact, this upper bound on $C(\sigma \mid [A]^\omega)$ is also optimal (up to a small error term).

\begin{restatable}{theorem}{lowerbound}\label{thm:ds_kolmogorov}
For any string $\sigma$ and set $A \subseteq \N$ of lower density at least $\delta \in (0, 1]$,
\[
  C(\sigma \mid [A]^\omega) \geq C^{0'}(\sigma) - \log(1/\delta) - O(\log\log(1/\delta))
\]
where the constant hidden by the $O(\cdot)$ notation does not depend on $\sigma$ or $A$.
\end{restatable}

We will prove these two results in Section~\ref{sec:kolmogorov} and comment on some further questions around how much information can be encoded into all infinite subsets of a dense set.

\subsection*{Acknowledgements}

We thank Andrew Marks for posing the question that motivated this paper and for several helpful conversations and Damir Dzhafarov for a useful conversation on the topic of Weihrauch reducibility.

\section{Preliminaries}

\subsection{Notation}

We will use the following notation for finite sets of natural numbers. For natural numbers $n < m$,
\begin{align*}
  [n] &= \{0, 1, \ldots, n\}\\
  [n, m) &= \{n, n + 1,\ldots, m - 1\}\\
  (n, m) &= \{n + 1, n + 2, \ldots, m - 1\}.
\end{align*}
We will also use the following notation related to infinite sets of natural numbers. For a set $A \subseteq \N$,
\begin{align*}
  A(n) &= \text{ the $n^\text{th}$ bit of $A$, i.e.\ $0$ if $n \notin A$ and $1$ if $n \in A$}\\
  \bar{A} &= \text{ the complement of $A$}\\
  [A]^\omega &= \text{ the set of infinite subsets of $A$}.
\end{align*}

By a \emph{Turing functional} we mean a program $\Phi$ with oracle access which has inputs in $\N$ and outputs in $\{0, 1\}$. We will use the following notation related to Turing functionals. For a Turing functional $\Phi$, oracle $A \subseteq \N$ and number $n \in \N$,
\begin{align*}
  \Phi(A, n) &= \text{ the output of $\Phi$ with oracle $A$ on input $n$}\\
  \Phi(A) &= \text{ the partial function $\N \to \{0, 1\}$ given by $n\mapsto \Phi(A, n)$}\\
  \Phi(A, n) \neq b &\; \text{ means that either $\Phi(A, n)$ diverges or $\Phi(A, n)\conv \neq b$}.
\end{align*}

We will sometimes want to consider a finite set $s \subseteq \N$ as an initial segment of an oracle. We will use $\Phi(s, n)$ to mean the output of $\Phi$ on input $n$ when run for at most $\max(s)$ steps, using $s$ as an oracle and automatically diverging if there is any query to the oracle about a number larger than $\max(s)$. 

Similarly, we will sometimes want to consider the Kolmogorov complexity of a string $\sigma$ relative to a finite set $s \subseteq \N$, which we will denote by $C^s(\sigma)$. More precisely, $C^s(\sigma)$ will denote the length of the shortest program which outputs $\sigma$ when using an oracle for $s$ and which runs for at most $\max(s)$ steps, never makes an oracle query about a number larger than $\max(s)$ and has length at most $\max(s)$. Note that it is possible that no such program exists due to the limitations on length and running time. In this case we define $C^s(\sigma) = \infty$. Also note that $C^s(\sigma)$ is computable.

\subsection{Density of sets of natural numbers}
\label{sec:prelim_density}

If $A \subseteq \N$ is a set of natural numbers then its \term{lower density}, denoted $\lowerd$, and \term{upper density}, denoted $\upperd$, are defined by
\begin{align*}
  \lowerd(A) &= \liminf_{n \to \infty} \frac{|A\cap[n]|}{n + 1}\\
  \upperd(A) &= \limsup_{n \to \infty} \frac{|A\cap[n]|}{n + 1}.
\end{align*}
If $\lowerd(A) > 0$ then we say $A$ has \term{positive lower density}. Similarly, if $\upperd(A) > 0$ then $A$ has \term{positive upper density}.

In order to work with dense subsets of $\N$, it will be helpful to introduce some auxiliary terminology. To motivate this terminology, note that if the upper density of a set $A$ is strictly greater than $\delta$ then there are infinitely many $n$ such that
\[
  |A\cap [n]| \geq \delta \cdot (n + 1).
\]
It is useful to be able to speak about the collection of all such $n$. To that end, we introduce the following terminology. For $A, D \subseteq \N$ and $\delta > 0$:
\begin{itemize}
\item $A$ is \term{$\delta$-dense at $n$} if $|A\cap [n]| \geq \delta \cdot (n + 1)$.
\item $A$ is \term{$\delta$-dense} if it is $\delta$-dense at every $n \in \N$ and \term{dense} if it is $\delta$-dense for some $\delta > 0$.
\item $A$ is \term{$\delta$-dense along $D$} if it is $\delta$-dense at every $n \in D$ and \term{dense along $D$} if it is $\delta$-dense along $D$ for some $\delta > 0$.
\end{itemize}

The reason these notions are useful is that they act as lower complexity versions of the properties of having positive upper or lower density.

For example, a set $A$ has positive lower density if and only if $\{0\}\cup A$ is $\delta$-dense for some $\delta > 0$ (though note that $\delta$ may need to be much lower than $\lowerd(A)$), but the property of having positive lower density is $\Sigma^0_2$, while the property of being $\delta$-dense is $\Pi^0_1$.

Similarly, a set $A$ has positive upper density if and only if it is $\delta$-dense along $D$ for some $\delta > 0$ and infinite set $D \subseteq \N$ (and this time, $\delta$ can be arbitrarily close to $\upperd(A)$), but the property of having positive upper density is $\Sigma^0_3$ while the property of being $\delta$-dense along $D$ is again $\Pi^0_1$.

\subsection{Mathias forcing}

Mathias forcing is a useful tool for constructing an infinite subset of a set while ensuring that the subset being constructed satisfies various properties. We will briefly review the basics of Mathias forcing; for a more complete introduction, see~\cite{hirschfeldt2015slicing}, Section 6.5.

A \term{condition for Mathias forcing} is a pair $(s, A)$ consisting of a finite set $s \subseteq \N$, called the \term{stem}, and an infinite set $A \subseteq \N$, called the \term{reservoir}, such that $\max(s) < \min(A)$. Often, the reservoir $A$ is required to come from some restricted class of sets, such as a Turing ideal or co-cone.

A Mathias condition, $(s, A)$, is extended by a condition, $(s', A')$, written $(s, A) \geq (s', A')$, if all of the following hold.
\begin{itemize}
\item $s \subseteq s'$
\item $A \supseteq A'$
\item and for all $n \in s' \setminus s$, $n \in A$.
\end{itemize}
In other words, $(s', A')$ is formed from $(s, A)$ by choosing finitely many elements of the reservoir to add to the stem and by removing some elements (possibly infinitely many) from the reservoir.

A Mathias condition, $(s, A)$, should be thought of as partially specifying a subset $G$ of $\N$ as follows: the stem $s$ consists of numbers that have already been put into $G$ and the reservoir $A$ consists of numbers that may be put into $G$ at some later stage. Thus any number which is not in $s \cup A$ is definitely not in $G$.

This can be made precise as follows. Any filter for Mathias forcing can be used to define a subset of $\N$ by taking the union of all the stems in the filter. We will denote this subset by $G$ and, in a slight abuse of terminology, will often conflate it with the filter itself (i.e.\ we will refer to $G$ as the generic).

It is straightforward to check that if the filter is sufficiently generic then $G$ is infinite. Also, say that a set $B$ is \term{compatible} with a condition $(s, A)$ if $s \subseteq B \subseteq s \cup A$. If $(s, A)$ is any element of the filter defining $G$ then it is straightforward to check that $G$ is compatible with $(s, A)$.

Thus one way (though not the only way) to show that every sufficiently generic $G$ satisfies some property $P$ is to show that the following set of conditions
\[
  \{(s, A) \mid \text{all infinite sets $C$ compatible with $(s, A)$ satisfy $P$}\}
\]
is dense in the Mathias forcing partial order.

\subsection{Seetapun's theorem}

In the introduction, we mentioned that our proof of Theorem~\ref{thm:main} uses Seetapun's theorem.

\seetapun*

We will also need to use a couple corollaries of this theorem.

\begin{corollary}
\label{cor:seetapun1}
Suppose $A \subseteq \N$ is infinite and does not compute $X$. Then for any $B \subseteq A$ there is some infinite subset of either $B$ or $A \setminus B$ which does not compute $X$.
\end{corollary}

\begin{proof}
Let $\pi \colon A \to \N$ be a bijection which is computable from $A$. The idea is to consider Seetapun's theorem applied to $\pi(B)$.

Note that since $\pi$ is a bijection, $\N \setminus \pi(B) = \pi(A \setminus B)$. Thus by Seetapun's theorem relativized to $A$, either $\pi(B)$  or $\pi(A\setminus B)$ has an infinite subset $C$ such that $C\oplus A$ does not compute $X$. Since $\pi$ is a bijection, $\pi^{-1}(C)$ is an infinite subset of either $B$ or $A\setminus B$. And since $A$ computes $\pi$, $C\oplus A$ computes $\pi^{-1}(C)$ and hence $\pi^{-1}(C)$ does not compute $X$.
\end{proof}

\begin{corollary}
\label{cor:seetapun2}
Suppose $A \subseteq \N$ is infinite and does not compute $X$. Then for any finite partition $B_1,\ldots,B_n$ of $A$, at least one of the $B_i$'s has an infinite subset which does not compute $X$.
\end{corollary}

\begin{proof}
The idea of the proof is to use induction on $n$ and to reduce the inductive case to the previous corollary.

The base case, $n = 1$ is trivial since in this case $B_1 = A$ and thus $B_1$ itself does not compute $X$.

Now assume for induction that the statement holds for all $A$, $X$ and partitions of length at $n$. Fix $A$ and $X$ as in the statement of the corollary and suppose $B_1,\ldots, B_{n + 1}$ is a partition of $A$. By applying Corollary~\ref{cor:seetapun1} to the sets $\bigcup_{i \leq n}B_i$ and $B_{n + 1}$, we obtain an infinite set $C$ which is a subset of either $\bigcup_{i \leq n}B_i$ or $B_{n + 1}$ such that $C$ does not compute $X$.

If $C \subseteq B_{n + 1}$ then we are done. If $C \subseteq \bigcup_{i \leq n}B_i$ then we can apply the inductive assumption to the partition $B_1\cap C, \ldots, B_n \cap C$ of $C$ to obtain an infinite set $D$ which does not compute $X$ and which is a subset of $B_i\cap C$ (and thus of $B_i$ as well) for some $i \leq n$.
\end{proof}

\section{Proof of the main theorem}
\label{sec:main}

For the rest of this section, suppose that $X$ is uncomputable and $A \subseteq \N$ has positive lower density. We want to construct an infinite subset of $A$ which does not compute $X$. We will construct this set using Mathias forcing. The basic strategy is as follows: let $G$ be a Mathias generic which is compatible with the condition $(\0, A)$ (i.e.\ the generic filter used to define $G$ contains the condition $(\0, A)$). Then we have
\begin{itemize}
    \item Since $G$ is compatible with $(\0, A)$, $G \subseteq A$.
    \item Since $G$ is sufficiently generic, $G$ is infinite.
\end{itemize}
If we could also show that any sufficiently generic $G$ does not compute $X$ then we would be done.

However, it is not hard to see that a generic for plain Mathias forcing---with no restrictions on the possible reservoirs---does not have this property. For example, suppose $B$ is an infinite set such that all infinite subsets of $B$ compute $X$. Then any Mathias generic $G$ which is compatible with $(\0, B)$ will compute $X$. To solve this problem, we will impose a restriction on the reservoirs of the Mathias conditions which guarantees that every sufficiently generic $G$ does not compute $X$.

\subsection{A natural idea that doesn't work}

Perhaps the most obvious restriction to put on the reservoirs is to require them to have positive lower density. Here's why this seems natural. The problem with plain Mathias forcing is that there are conditions $(s, B)$ such that \emph{every set} compatible with $(s, B)$ computes $X$. However, if we believe the statement we are trying to prove---that every set of positive lower density contains an infinite subset which does not compute $X$---then this same problem cannot occur when $B$ is required to have positive lower density.

However, this approach does not work. The problem, briefly stated is that even though every set of positive lower density contains an infinite subset which does not compute $X$ (as we will eventually show), there is a set of positive lower density, all of whose subsets of positive lower density uniformly compute $X$. Thus if we only use reservoirs that have positive lower density, then given a condition $(s, B)$ and a Turing functional $\Phi$, there is no obvious way to find a condition $(s', B') \leq (s, B)$ forcing that the generic $G$ does not compute $X$ via $\Phi$ (since it could be that for any such $(s', B')$, $s'\cup B'$ itself computes $X$ via $\Phi$).

In fact, it is even possible to show that if we only use reservoirs that have positive lower density then it is possible that $G$ is forced to compute $X$. This can be shown using the following proposition, which we will prove in Section~\ref{sec:dense_subsets}.

\begin{restatable*}{proposition}{densesubsetsofdense}
For any set $X$, there is some set $A \subseteq \N$ of positive lower density such that all subsets of $A$ of positive lower density compute $X$ uniformly.
\end{restatable*}

\begin{proposition}
Let $A$ and $X$ be as in the above proposition. If $G$ is generic for Mathias forcing with reservoirs of positive lower density and $G$ is compatible with $(\0, A)$ then $G$ computes $X$.
\end{proposition}

\begin{proof}
Let $\Phi$ be a Turing functional witnessing the fact that all subsets of $A$ of positive lower density compute $X$ uniformly (in other words, such that for all $B \subseteq A$ of positive lower density, $\Phi(B) = X$). For each $n$, we will show that the set of conditions
\[
  \D_n = \{(s, B) \mid \Phi(s, n)\conv = X(n)\}
\]
is dense below $(\0, A)$. This shows that for each $n$, $\Phi(G, n)\conv = X(n)$ and hence that $\Phi(G) = X$.

So fix $n$ for which we will show that $\D_n$ is dense below $(\0, A)$. Let $(s, B)$ be an arbitrary condition extending $(\0, A)$. We need to show that there is some condition $(s', B')$ extending $(s, B)$ such that $\Phi(s', n)\conv = X(n)$.

Since $s\cup B$ is a subset of $A$ of positive lower density, we must have $\Phi(s\cup B, n)\conv = X(n)$. Let $s'$ be an initial segment of $s\cup B$ which is long enough to witness this and set $B' = B\setminus s'$. Then $(s', B') \leq (s, B)$ and $\Phi(s', n)\conv = X(n)$ as desired.
\end{proof}

\subsection{Density Mathias forcing}

We have seen that if we want to use Mathias forcing to construct an infinite subset of $A$ which does not compute $X$ then neither allowing all infinite sets as reservoirs nor requiring the reservoirs to have positive lower density works. The problem with the former is that there are too many reservoirs available, in particular there are reservoirs whose infinite subsets all compute $X$. The problem with the latter is that there are too few reservoirs available, in particular, for a fixed reservoir $B$ and Turing functional $\Phi$ it may not be possible to find a reservoir $B' \subseteq B$ witnessing that not all infinite subsets of $B$ compute $X$ via $\Phi$.
% ---for a fixed program $\Phi$ we cannot always find a reservoir forcing the generic to not compute $X$ via $\Phi$. 
Thus we want a requirement that lies in-between these two extremes.

We will now describe such a requirement. In essence, we will allow a set $B$ to be a reservoir if it has positive upper density and the fact that it has positive upper density is witnessed by a set that does not compute $X$. We will refer to Mathias forcing with reservoirs satisfying this condition as \term{density Mathias forcing}.

\begin{definition}
A condition for \term{density Mathias forcing} is a Mathias condition $(s, B)$ for which there is some set $D$ such that
\begin{enumerate}
\item $B$ is dense along $D$
\item and $D$ does not compute $X$.
\end{enumerate}
\end{definition}

Note that $(\0, A)$ itself is a density Mathias condition: $A$ is dense along $\N\setminus [\min(A)]$ and since $X$ is uncomputable, $\N\setminus [\min(A)]$ does not compute $X$.

\subsection{The proof}

We will now show that if $G$ is sufficiently generic for density Mathias forcing then $G$ does not compute $X$. We will begin with a technical lemma which we will use in the proof. Then we will prove the main lemma, which shows that for a single program $\Phi$ there is a dense set of conditions which guarantee that $\Phi(G) \neq X$.

\begin{lemma}
\label{lemma:forcing_partition}
If $(s, B)$ is a density Mathias condition and $B_1,\ldots, B_k$ is a finite partition of $B$ then for some $i \leq k$, $(s, B_i)$ is a density Mathias condition.
\end{lemma}

\begin{proof}
Since $(s, B)$ is a density Mathias condition there is some $\delta > 0$ and some infinite set $D$ such that $B$ is $\delta$-dense along $D$ and $D$ does not compute $X$. Now define subsets $D_1,\ldots, D_k$ of $D$ by
\[
  D_i = \{n \in D \mid i \text{ is least such that $B_i$ is $\delta/k$-dense at $n$}\}.
\]

We claim that $D_1,\ldots,D_k$ partition $D$. To see why, let $n \in D$. Thus $B\cap [n] \geq \delta\cdot(n + 1)$. Since $B_1,\ldots,B_k$ partition $B$, at least one of the sets
\[
  B_1\cap[n], B_2\cap[n], \ldots, B_k\cap[n]
\]
must have size at least $\delta\cdot(n + 1)/k$. Thus there is some $i$ such that $B_i$ is $\delta/k$-dense at $n$.

Since $D$ does not compute $X$ and $D_1,\ldots, D_k$ partition $D$, we can apply Corollary~\ref{cor:seetapun2} to Seetapun's theorem to get an infinite set $E$ which is contained in some $D_i$ and which does not compute $X$. It is straightforward to check that $B_i$ is $\delta/k$-dense along $E$ and thus $(s, B_i)$ is a density Mathias condition.
\end{proof}

\begin{lemma}
\label{lemma:main}
For any Turing functional $\Phi$ and density Mathias condition $(s, B)$, there is some density Mathias condition $(s', B') \leq (s, B)$ such that for any set $C$ compatible with $(s', B')$, $\Phi(C) \neq X$.
\end{lemma}

\begin{proof}
Since $(s, B)$ is a density Mathias condition, we can fix some $\delta > 0$ and infinite set $D$ such that $B$ is $\delta$-dense along $D$ and $D$ does not compute $X$. There are two cases to consider.
\begin{enumerate}
\item It is possible to make $\Phi$ wrong on some input: there is some $n \in \N$ and finite set $t\subseteq B$ such that $\Phi(s\cup t, n)\conv \neq X(n)$.
\item It is impossible to make $\Phi$ wrong on any input: for every $n \in \N$ and finite set $t\subseteq B$, $\Phi(s\cup t, n)$ either diverges or is equal to $X(n)$.
\end{enumerate}
The first case is easy: we can simply take $s' = s\cup t$ and $B' = B\setminus [\max(t)]$. Thus we may assume we are in the second case.

Since we are in the second case, it follows that for every $C$ compatible with $(s, B)$, $\Phi(C)$ never disagrees with $X$---though it may diverge on some inputs. Our goal is to find $(s', B')$ extending $(s, B)$ such that for every $C$ compatible with $(s', B')$, $\Phi(C)$ does diverge on at least one input.

\medskip\noindent\textbf{Intuition.}
Here's the basic idea of the proof. By using ideas from proofs of Seetapun's theorem, it is not too hard to find a set $B_0$ and number $n_0$ such that for all $C$ compatible with $(s, B_0)$, $\Phi(C, n_0) \neq X(n_0)$. It is also not too hard to ensure $B_0$ is $\delta$-dense along $D$. We would like to take $B' = B\cap B_0$. The reason is that any $C$ compatible with $(s, B\cap B_0)$ is compatible with both $(s, B)$ and $(s, B_0)$ and thus for such a $C$, $\Phi(C, n_0)$ is neither different from $X(n_0)$ (because $C$ is compatible with $(s, B)$) nor equal to $X(n_0)$ (because $C$ is compatible with $(s, B_0)$) and thus $\Phi(C, n_0)$ must diverge.

The problem with this idea is that $B\cap B_0$ might not be very dense---in fact, it might even be empty. So instead, we will iterate: we will find another set $B_1$ and number $n_1$ which have the same properties as $B_0$ and $n_0$ but such that $B_1$ is disjoint from $B_0$. If neither $B\cap B_0$ nor $B\cap B_1$ work then we will keep going and find $B_2$ disjoint from both, and so on.

The fact that sets are disjoint and all fairly dense will ensure that this process cannot go on forever: it is not possible to have more than $1/\delta$ disjoint sets which are all $\delta$-dense. Thus we will eventually find some $B_i$ which works. In practice, carrying out this idea involves a lot of additional technical details.

\medskip\noindent\textbf{General strategy.}
We will find disjoint sets $B_0, \ldots, B_k$ and numbers $n_0, \ldots, n_k$ such that
\begin{enumerate}
\item for each $i \leq k$ and all $C$ compatible with $(s, B_i)$, $\Phi(C, n_i) \neq X(n_i)$ (i.e.\ it may diverge or it may converge to some output that is not equal to $X(n_i)$)
\item and $(s, B\cap (B_0\cup\ldots\cup B_k))$ is a density Mathias condition.
\end{enumerate}
Before explaining how to find the $B_i$'s, let's explain why this is enough to finish the proof. Since $(s, B\cap (B_0\cup \ldots \cup B_k))$ is a density Mathias condition whose reservoir is partitioned by $B\cap B_0,\ldots,B\cap B_k$, Lemma~\ref{lemma:forcing_partition} implies that there is some $i$ such that $(s, B\cap B_i)$ is a density Mathias condition.

We now claim that we can take $s' = s$ and $B' = B\cap B_i$. To see why, note that if $C$ is compatible with $(s, B\cap B_i)$ then it is compatible with both $(s, B)$ and $(s, B_i)$. Thus $\Phi(C, n_i)$ can neither disagree with $X(n_i)$ (since $C$ is compatible with $(s, B)$) nor agree with $X(n_i)$ (since $C$ is compatible with $(s, B_i)$) and therefore $\Phi(C, n_i)$ must diverge.

\medskip\noindent\textbf{Strategy to construct $\bm{B_0, \ldots, B_k}$.} 
We will build the sequence of $B_i$'s inductively. To make the induction work, we will require the sets $B_i$ to satisfy some additional properties. In particular, we will construct a sequence of sets $B_0, B_1, B_2,\ldots$ along with sets $D \supseteq D_0 \supseteq D_1 \supseteq D_2 \supseteq \ldots$ and numbers $n_0, n_1, n_2,\ldots$ such that for each $i$,
\begin{enumerate}
\item $D_i$ is infinite.
\item $B_i$ is $\delta/2$-dense along $D_i$.
\item $B_0\oplus B_1\oplus\ldots\oplus B_i\oplus D_i$ does not compute $X$.
\item $B_i$ is disjoint from $B_0\cup\ldots\cup B_{i - 1}$
\item For each $C$ compatible with $(s, B_i)$, $\Phi(C, n_i) \neq X(n_i)$.
\end{enumerate}
We will show by induction that if we have built sequences $B_0,\ldots,B_k$, $D_0,\ldots, D_k$ and $n_0,\ldots,n_k$ satisfying these requirements then either $(s, B\cap (B_0\cup\ldots\cup B_k))$ is a density Mathias condition (and thus we are finished building the sequence) or we can extend the sequence---i.e.\ find $B_{k + 1}, D_{k + 1}$ and $n_{k + 1}$ satisfying the requirements. To finish, we will show that any sequence satisfying the requirements cannot be infinite.

\medskip\noindent\textbf{Extending the sequence: picking $\bm{D_{k + 1}}$.}
Suppose that we have built sequences $B_0,\ldots,B_k$, $D_0,\ldots, D_k$ and $n_0,\ldots,n_k$ satisfying the five requirements listed above. If we knew that $B\cap(B_0\cup\ldots\cup B_k)$ was dense along $D_k$ then we could stop: in that case $(s, B\cap(B_0\cup\ldots\cup B_k))$ would be a density Mathias condition. More generally, if we knew that $B\cap(B_0\cup\ldots\cup B_k)$ was dense along any infinite subset of $D_k$ which does not compute $X$ then we could stop. So let's assume that's not the case and show we can extend the sequence $B_0,\ldots, B_k$.

Let $E$ be the subset of $D_k$ defined by
\[
  E = \{n \in D_k \mid B\cap (B_0 \cup \ldots \cup B_k) \text{ is $\delta/2$-dense at } n\}.
\]
By (the relativized form of) Corollary~\ref{cor:seetapun2} to Seetapun's theorem, there is an infinite subset $D_{k + 1}$ of either $E$ or $D_k \setminus E$ such that $B_0\oplus\ldots \oplus B_k\oplus D_{k + 1}$ does not compute $X$. However, $D_{k + 1}$ cannot be a subset of $E$ since then $B\cap(B_0\cup\ldots\cup B_k)$ would be dense along $D_{k + 1}$, which contradicts our assumption above.

Thus $D_{k + 1} \subseteq D_k\setminus E$. This implies $B\setminus (B_0\cup\ldots\cup B_k)$ is $\delta/2$-dense along $D_{k + 1}$. To see why, consider any $n \in D_{k + 1}$. Since $D_{k + 1} \subseteq D_k \subseteq D$ and $B$ is $\delta$-dense along $D$, $B$ is $\delta$-dense at $n$. Since $n \notin E$, $B\cap (B_0\cup\ldots\cup B_k)$ is not $\delta/2$-dense at $n$. Therefore $B\setminus (B_0\cup\ldots\cup B_k)$ must be $\delta/2$-dense at $n$.

\medskip\noindent\textbf{Extending the sequence: picking $\bm{B_{k + 1}}$.}
Observe that $B\setminus (B_0\cup\ldots\cup B_k)$ has the following properties.
\begin{enumerate}
\item As we just showed, it is $\delta/2$-dense along $D_{k + 1}$.
\item It is disjoint from $B_0\cup\ldots\cup B_k$.
\item Since it is a subset of $B$, for every finite subset $t$ and every $n \in \N$, $\Phi(s\cup t, n)$ either diverges or is equal to $X(n)$.
\end{enumerate}
These properties look almost like what we want from $B_{k + 1}$ but in the last item above we would like to have ``not equal to $X(n)$'' rather than ``equal to $X(n)$.'' In other words, we want to find a set which is very similar to $B\setminus (B_0\cup\ldots\cup B_k)$ but differs in one key respect. We will argue that if we cannot find such a set then $X$ is computable from $B_0\oplus \ldots\oplus B_k\oplus D_{k + 1}$.

For each $n \in \N$ and $b \in \{0, 1\}$, define a set $\CC_{n, b} \subseteq \powerset(N)$ by
\begin{align*}
  \CC_{n, b} = \{ Y \subseteq \N \mid\;
  &Y \text{ is disjoint from } B_0\cup\ldots\cup B_k\\
  &\text{and } Y \text{ is $\delta/2$-dense along } D_{k + 1}\\
  &\text{and for all finite $t\subseteq Y$, either $\Phi(s\cup t, n)\!\uparrow$ or $\Phi(s\cup t, n)\conv = b$} \}.
\end{align*}
Note that each $\CC_{n, b}$ is a $\Pi^0_1$ class relative to $B_0\oplus\ldots\oplus B_k\oplus D_{k + 1}$. Thus for each $(n, b)$, there is some $B_0\oplus\ldots\oplus B_k\oplus D_{k + 1}$-computable binary tree $T_{n, b}$ whose paths are exactly the elements of $\CC_{n, b}$. Furthermore, it is easy to see that $T_{n, b}$ can be computed uniformly in $(n, b)$.

We would now like to show that for some $n$ and $b \neq X(n)$, $\CC_{n, b}$ is nonempty. Suppose not. Then for each $n \in \N$ and $b\in\{0, 1\}$, we have the following
\begin{enumerate}
\item If $X(n) = b$ then $\CC_{n, b}$ is nonempty (as witnessed by $B\setminus (B_0\cup\ldots\cup B_k)$) and thus $T_{n, b}$ is infinite.
\item If $X(n) \neq b$ then $\CC_{n, b}$ is empty and so, by K\"onig's lemma, $T_{n, b}$ is finite.
\end{enumerate}
Therefore to compute $X(n)$ using $B_0\oplus\ldots\oplus B_k\oplus D_{k + 1}$, we simply need to check which of $T_{n, 0}$ and $T_{n, 1}$ is finite.

Thus there is some $n$ and $b \neq X(n)$ such that $\CC_{n, b}$ is nonempty. By the cone avoiding basis theorem, relativized to $B_0\oplus\ldots\oplus B_k\oplus D_{k + 1}$, we can find some $B_{k + 1} \in \CC_{n, b}$ such that
\[
\left(B_0\oplus\ldots\oplus B_k\oplus D_{k + 1}\right)\oplus B_{k + 1} \ngeq_T X.
\]
It is straightforward to check that this $B_{k + 1}$ satisfies all the necessary requirements.

\medskip\noindent\textbf{What about the base case?}
Since we are forming the sequences $B_0, B_1,\ldots$ and $D_0,D_1,\ldots$ inductively, it would seem that we need to explain the base case: how to pick $B_0$, $D_0$ and $n_0$. However, the argument we gave in the inductive case to pick $B_{k + 1}$ and $D_{k + 1}$ works just as well for the base case, with the exception that instead of having to pick $D_0$ carefully we can just take $D_0 = D$. Also in this case, we should interpret the union $B_0\cup\ldots\cup B_k$ as the empty set.

\medskip\noindent\textbf{The sequence cannot be infinite.}
We have shown that if we have formed sets $B_0,\ldots,B_k$ and $D_0,\ldots,D_k$ satisfying the five conditions listed above then either $(s, B\cap(B_0\cup\ldots\cup B_k))$ is a density Mathias condition (in which case we are done) or we can find $B_{k + 1}$ and $D_{k + 1}$ extending the sequence. But how do we know the first option ever holds? In other words, how do we know we cannot just extend the sequence indefinitely? We will now show that no sequence satisfying the five requirements can be longer than $2/\delta$.

Suppose for contradiction that $B_0,\ldots,B_k$ and $D_0,\ldots,D_k$ do satisfy the five requirements and $k > 2/\delta$. Let $n$ be any element of $D_k$. Note that for each $i \leq k$, $D_k \subseteq D_i$ and hence $B_i$ is $\delta/2$-dense at $n$. In other words,
\[
  |B_i\cap [n]| \geq \frac{\delta}{2}\cdot(n + 1).
\]
However, since the $B_i$'s are all disjoint, this gives us $k > 2/\delta$ disjoint subsets of $[n]$, all of size at least $\delta/2\cdot(n + 1)$, which is impossible.
\end{proof}

We can now prove Theorem \ref{thm:main}.

\main*

\begin{proof}
Let $G$ be a generic for density Mathias forcing which is compatible with the condition $(\0, A)$. We claim that $G$ is an infinite subset of $A$ which does not compute $X$. Since $G$ is generic, it is infinite and since $G$ is compatible with $(\0, A)$, $G \subseteq A$. Now let $\Phi$ be any Turing functional. By Lemma~\ref{lemma:main}, the following set of density Mathias conditions is dense
\[
  \{(s, B) \mid \text{for all $C$ compatible with $(s, B)$, $\Phi(C) \neq X$}\}
\]
and so $\Phi(G) \neq X$. Since this holds for all $\Phi$, $G$ does not compute $X$.
\end{proof}

\subsection{An open question}

Monin and Patey have proved the following variation on Seetapun's theorem, in which computability is replaced by hyperarithmetic reducibility~\cite{monin2021weakness}.

\begin{theorem}[Monin and Patey]
For any non-hyperarithmetic set $X$ and any set $A \subseteq \N$, there is an infinite subset $B$ of either $A$ or $\bar{A}$ such that $X$ is not hyperarithmetic in $B$.
\end{theorem}

It seems natural to ask whether the analogous variation on Theorem~\ref{thm:main} is true. Answering this question seems to be beyond the techniques we used to prove our theorem.

\begin{question}
Suppose $X$ is not hyperarithmetic and $A \subseteq \N$ has positive lower density. Must $A$ have an infinite subset $B$ such that $X$ is not hyperarithmetic in $B$?
\end{question}

\section{The main theorem is sharp}
\label{sec:sharp}

Our goal in proving Theorem~\ref{thm:main} was to formalize the intuition that if every infinite subset of a set $A$ can compute an uncomputable set $X$ then $A$ must be sparse. In our theorem, we interpreted sparse to mean lower density zero. However, this is not the only reasonable interpretation of what it means for a set of natural numbers to be sparse. In this section, we will consider a couple other notions of sparsity and give counterexamples showing that for each, our main theorem becomes false.

In other words, for various notions of sparsity, we will prove that there is an uncomputable $X$ and a set $A \subseteq \N$ which is not sparse such that all infinite subsets of $A$ compute $X$. In nearly all cases, we will be able to strengthen the counterexample in two ways.

First, instead of just finding a single $X$ that can be encoded into a non-sparse set in this way, we will show that any set $X$ can be so encoded---i.e.\ for every $X$ there is some non-sparse $A \subseteq \N$ such that all infinite subsets of $A$ compute $X$.

Second, instead of every subset of $A$ simply computing $X$, we can ensure that they all compute $X$ \term{uniformly}---i.e.\ there is a single Turing functional $\Phi$ such that for all infinite subsets $B$ of $A$, $\Phi(B) = X$.

\subsection{Sets of positive upper density}

Our first counterexample concerns replacing positive lower density with positive upper density. In this case, we can even find a counterexample where $A$ has upper density one.

\begin{proposition}
For any set $X$, there is some set $A\subseteq \N$ such that $\upperd(A) = 1$ and all infinite subsets of $A$ compute $X$ uniformly.
\end{proposition}

\begin{proof}
Let $\tilde{A}$ be any infinite set, all of whose infinite subsets compute $X$ uniformly (for example, $\tilde{A}$ could be the set produced by the Dekker-Myhill method explained in the introduction). The idea is to build $A$ by copying $\tilde{A}$, but add in a lot of redundancy so that $A$ is occasionally very dense.

To that end, pick a computable sequence of numbers $n_0 < n_1 < n_2 < \ldots$ which grows fast enough that $n_{i + 1} > i\cdot n_i$ and define
\[
 A = \bigcup_{i \in \tilde{A}}[n_i, n_{i + 1}).
\]
In other words, for each $i \in \tilde{A}$, $A$ contains the entire interval $[n_i, n_{i + 1})$.

First observe that $A$ has upper density $1$. Indeed, for each $i \in \tilde{A}$, $A$ is $\frac{i - 1}{i}$-dense at $n_{i + 1}$.

Second, observe that all infinite subsets of $A$ compute $X$ uniformly: given an infinite subset $B\subseteq A$, we can uniformly compute the following infinite subset $\tilde{B}$ of $\tilde{A}$,
\[
  \tilde{B} = \{ i \mid \exists n \in [n_i, n_{i + 1})\cap A\},
\]
and then use $\tilde{B}$ to compute $X$.
\end{proof}

In spite of this counterexample, our main theorem can be strengthened to hold for certain sets of positive upper density.

\begin{proposition}
For any uncomputable $X$ and set $A\subseteq \N$ such that $A$ is dense along some infinite set $D$ which does not compute $X$, there is some infinite subset of $A$ which does not compute $X$.
\end{proposition}

\begin{proof}
Note that $(\0, A)$ is a density Mathias condition, as witnessed by $D$. Thus there is a generic $G$ for density Mathias forcing which is compatible with $(\0, A)$. Exactly as in the proof of Theorem~\ref{thm:main}, $G$ is an infinite subset of $A$ which does not compute $X$.
\end{proof}

\subsection{Sets whose density goes to zero slowly}

Our next counterexample concerns sets of lower density zero where the density approaches $0$ very slowly. Define the \term{density function} of a set $A$ to be the function
\[
  d_A(n) = \frac{|A\cap[n]|}{n + 1}.
\]

Note that if a set $A$ has positive lower density then $d_A(n)$ is bounded away from $0$. By contrast, most of the methods we have seen so far of coding a set $X$ into all infinite subsets of a set produce sets whose density functions converge to $0$ very rapidly. For example, the Dekker-Myhill method we mentioned in the introduction produces a set $A$ such that $d_A(n)$ is approximately $\frac{\log(n)}{n}$. In the previous subsection, we saw a method for which this is not quite true---in particular, we saw a method which produces a set $A$ of upper density one, which implies that $d_A(n)$ is close to $1$ infinitely often. However, even for this method, $d_A(n)$ is also infinitely often smaller than $\frac{\log(n)}{n}$.

% is---at least infinitely often---smaller than about $\frac{\log(n)}{n}$.
% By contrast, the Dekker-Myhill method of coding a set $X$ into all infinite subsets of a set produces a set $A$ whose density functions is approximately $\frac{\log(n)}{n}$. 
% Furthermore, above we gave an example of coding a set $X$ into all infinite subsets of a set $A$ of upper density one and the density function of this set $A$, while infinitely often close to $1$, is also infinitely often smaller than $\frac{\log(n)}{n}$.

Based on these examples, one might guess that if all infinite subsets of a set $A$ compute an uncomputable set $X$ then $d_A(n)$ cannot be lower bounded by any monotone function that goes to zero much slower than $\frac{\log(n)}{n}$---in other words that there must be infinitely many places where the density of $A$ is exponentially small. However, this is false: we will show that it is possible to encode any set $X$ into all infinite subsets of a set $A$ whose density function goes to zero arbitrarily slowly.

\begin{definition}
For any function $f \colon \N \to [0, 1]$, a set $A \subseteq \N$ is \term{$f$-dense} if for all $n$,
\[
  d_A(n) \geq f(n).
\]
\end{definition}

\begin{proposition}
Suppose $f \colon \N \to [0, 1]$ is a function such that
\[
  \lim_{n \to \infty} f(n) = 0.
\]
Then for any set $X$ there is some $A$ such that $A$ is $f$-dense and all infinite subsets of $A$ compute $X$ uniformly.
\end{proposition}

\begin{proof}
The idea is to modify the Dekker-Myhill coding method explained in the introduction, but slow it down so that it produces an $f$-dense set. In Dekker and Myhill's scheme, each element of $A$ encodes one more bit of $X$ than the previous element. In our modified version, we will repeatedly encode the same number of bits of $X$ until $f$ becomes small enough to allow us to encode more. Roughly speaking, we will only start using elements of $A$ to encode the first $n$ bits of $X$ once $f$ drops below $1/2^n$. 

One comment before we go into the details of the construction: it might seem necessary to require $f$ to be computable but there is a trick that allows us to avoid requiring that. Essentially, we can use elements of $A$ to encode not just some bits of $X$ but also to encode how many bits are encoded.

To define $A$ formally, first pick a sequence $0 < n_1 < n_2 < \ldots$ which grows fast enough that for all $i$:
\begin{enumerate}
\item for all $m \geq n_i$, $f(m) \leq 1/(5\cdot 2^{2i + 1})$
\item and $n_{i + 1} - n_i$ is divisible by $2^{2i + 1}$.
\end{enumerate}
The idea is that in the interval $[n_i, n_{i + 1})$, each element of $A$ will encode the first $i$ bits of $X$. Note that the sequence $n_0, n_1, n_2, \ldots$ is not required to be computable.

Next, for each $i$, let $\sigma_i$ denote the length $2i + 1$ binary string consisting of the first $i$ bits of $X$, followed by a $1$, followed by a $i$ zeros---i.e.
\[
  \sigma_i = (X\uh i)\concat 1\concat \underbrace{00\ldots 0}_{i \text{ times}}.
\]
Now define $A$ as follows.
\begin{enumerate}
\item First put $[0, n_1)$ into $A$.
\item Next, for each $i \geq 1$ and $m \in [n_i, n_{i + 1})$, put $m$ into $A$ if the binary expansion of $m$ ends with the string $\sigma_i$.
\end{enumerate}
We will now show that $A$ has the properties desired.

\begin{claim*}
$A$ is $f$-dense.    
\end{claim*}

\begin{proof}
Consider a single interval $[n_i, n_{i + 1})$. Note that this interval is composed of some number of disjoint intervals of length $2^{2i + 1}$ and that each one of these smaller intervals contains exactly one element of $A$. Also note that $A$ contains all numbers less than $n_1$. From these two observations, it is easy to see that for each $i$, $A$ is $1/2^{2i + 1}$-dense at $n_{i + 1}$. In other words, $A$ is sufficiently dense at the endpoints of each interval $[n_i, n_{i + 1})$. It remains to check that $A$ is also sufficiently dense in the interior of these intervals.

Note that any number $m \in (n_i, n_{i + 1})$ can be written as $m = n_i + k\cdot 2^{2i + 1} + l$ for some $k$ and $l < 2^{2i + 1}$. Our observations above imply that $A$ is $1/2^{2i + 1}$-dense at $n_i + k\cdot 2^{2i + 1}$. We now want to show that $f(m)$ and $l$ are small enough that $A$ is still sufficiently dense at $m$.

To see why, note that by assumption, $n_i > 2^{2i - 1}$ and thus $l < 4\cdot(n_i + k\cdot 2^{2i + 1})$. Thus, since $A$ is $1/2^{2i + 1}$-dense at $n_i + k\cdot 2^{2i + 1}$, it is also $1/(5\cdot 2^{2i + 1})$-dense at $m$. By our choice of $n_i$, this implies that $A$ is $f(m)$-dense at $m$.
\end{proof}

\begin{claim*}
Every infinite subset of $A$ computes $X$ uniformly.
\end{claim*}

\begin{proof}
Suppose $B$ is an infinite subset of $A$. For any element $m \in B$ such that $m \geq n_1$, the binary expansion of $m$ must end with some number of zeros. Suppose it ends with $k$ zeros. Then the binary expansion of $m$ must end with a string $\sigma$ of length $k$, followed by a one, followed by $k$ zeros. By construction of $A$, this string $\sigma$ is guaranteed to be an initial segment of $X$. So to compute any bit $X(i)$ of $X$, we simply need to find an element $m \geq n_1$ of $B$ whose binary expansion ends with a long enough string of zeros. And such an element is guaranteed to exist by our construction of $A$ and the fact that $B$ is infinite.
\end{proof}

\noindent This concludes the proof.
\end{proof}

\subsection{Dense subsets of dense sets}
\label{sec:dense_subsets}

Our final counterexample is of a somewhat different nature than the first two. In our previous counterexamples, we considered possible strengthenings of our main theorem given by relaxing the density requirement on $A$. We will now consider possible strengthenings given by restricting which subsets of $A$ we are allowed to use.

The main theorem of this paper shows that it is impossible to encode an uncomputable set $X$ into all infinite subsets of a set $A$ of positive lower density. But what if we only want to encode $X$ into all subsets of $A$ of positive lower density? Or into all subsets of positive upper density? We will show that such encoding is possible in both cases, but that it can only be done uniformly in the first case.

\densesubsetsofdense

\begin{proof}
Suppose we only had to consider $1/2$-dense subsets of $A$. For each $n$, any such subset must contain at least one element of $A$ between $n$ and $2n$. Thus we could code one bit of $X$ into the parity of the elements of $A$ between $n$ and $2n$ and be sure that any $1/2$-dense subset of $A$ will be able to recover this bit. However, we want our procedure to work for subsets of $A$ of any constant density greater than zero, not just $1/2$-dense subsets. So we will encode each bit of $X$ infinitely many times, each time making a more conservative assumption about the density of the subset of $A$ doing the decoding.

More formally, pick a computable bijection $\pi \colon \N \to \N\times \Q^+$ and a computable sequence $0 = n_0 < n_1 < n_2 < \ldots$ such that for all $i$,
\begin{enumerate}
\item if $\pi(i) = (j, \delta)$ then $\delta\cdot n_{i + 1} > n_i$
\item and $n_i$ is even.
\end{enumerate}
The idea is that in the interval $[n_i, n_{i + 1})$, elements of $A$ will code the $j^\text{th}$ bit of $X$ in such a way that any $\delta$-dense subset of $A$ will be able to decode it.

Now define $A$ as follows. For each $i \in \N$ and $m \in [n_i, n_{i + 1})$, let $(j, \delta) = \pi(i)$ and put $m$ into $A$ if
\[
  \begin{cases}
    m \text{ is even and } X(j) = 0\\
    \text{or } m \text{ is odd and } X(j) = 1.
  \end{cases}
\]
In other words, in the interval $[n_i, n_{i + 1})$, the $j^\text{th}$ bit of $X$ is encoded into the parity of the elements of $A$.

First observe that $A$ has lower density $1/2$: for each $m$, either $2m$ or $2m + 1$ is in $A$ (this is why we wanted to make sure each $n_i$ is even).

Second, observe that all subsets of $A$ of positive lower density can compute $X$ uniformly. Given a subset $B \subseteq A$ of positive lower density and a $j \in \N$, we can compute $X(j)$ as follows: search for any $i$ such that $\pi(i) = (j, \delta)$ for some $\delta$ and $[n_i, n_{i + 1})\cap B$ is nonempty. Then from the parity of any element of $[n_i, n_{i + 1})\cap B$ we can recover $X(j)$. The fact that such an $i$ must exist follows from our choice of $n_i$'s and the fact that $B$ has positive lower density---if $B$ is $\delta$-dense and $\pi(i) = (j, \delta)$ then $B$ must have at least one element in the interval $[n_i, n_{i + 1})$ because otherwise $B$ cannot be $\delta$-dense at $n_{i + 1}$.
\end{proof}

For the case of subsets of positive upper density, the counterexample can be constructed using a result of Bienvenu, Day and H\"olzl~\cite{bienvenu2009bi}.

\begin{theorem}[Bienvenu, Day and H\"olzl]
For any set $X$, there is a set $A$ such that for all partial functions $f\colon \N \to \{0, 1\}$, if the domain of $f$ has positive upper density and for all $n \in \dom(f)$, $f(n) = A(n)$ then $f$ computes $X$.
\end{theorem}

\begin{corollary}
For any set $X$, there is a set $A$ of positive lower density such that all subsets of $A$ of positive upper density compute $X$.
\end{corollary}

\begin{proof}
Let $A$ be the set from Bienvenu, Day and H\"olzl's theorem. Let $\tilde{A} = \{2n \mid n \in A\} \cup \{2n + 1 \mid n \notin A\}$. Note that $\tilde{A}$ has lower density $1/2$ and that from each element of $\tilde{A}$, we can recover one bit of the characteristic function of $A$.

Thus any subset $B$ of $\tilde{A}$ can compute a partial function $f_B \colon \N \to \{0, 1\}$ which agrees with $A$ everywhere it is defined. Furthermore, the density of $\dom(f_B)$ at any $n$ is about twice the density of $B$ at $2n + 2$. Therefore if $B$ has positive upper density then so does $\dom(f_B)$. So for any such $B$, our choice of $A$ implies that $f_B$, and hence $B$ itself, computes $X$.
\end{proof}

As we noted above, it is impossible to strengthen this corollary to make all subsets of $A$ of positive upper density compute $X$ uniformly.

% in the example above, it is impossible to make all subsets of $A$ of positive upper density compute $X$ uniformly.

\begin{proposition}
For any uncomputable set $X$, set $A$ of positive lower density and Turing functional $\Phi$, there is some subset $B \subseteq A$ of positive upper density such that $\Phi(B) \neq X$.
\end{proposition}

\begin{proof}
This follows directly from Lemma~\ref{lemma:main}. In particular, using that lemma we can find a density Mathias condition $(s, B)$ extending $(\0, A)$ such that for every set $C$ compatible with $(s, B)$, $\Phi(C) \neq X$. In particular, if $C = s\cup B$ then $\Phi(C) \neq X$. Since $(s, B)$ is a density Mathias condition, $C$ has positive upper density and so we are done.
\end{proof}

\section{Avoiding PA degree}
\label{sec:main-PA}

In this section we prove the theorem promised in the introduction which modifies our main theorem to avoid PA degree rather than cones of Turing degrees.

\mainPA*

In order to prove this theorem, it will be convenient to use the following alternative definition of PA degree.

\begin{definition}
Given a partial function $f \colon \N \to \{0, 1\}$, a \term{completion} of $f$ is a total function $g \colon \N \to \{0, 1\}$ such that for every $n$ in the domain of $f$, $f(n) = g(n)$.
\end{definition}

\begin{definition}
A set $A$ is said to be of \term{PA degree} if for every partial computable function $f \colon \N \to \{0, 1\}$, $A$ computes a completion of $f$.
\end{definition}

We can thus obtain Theorem~\ref{thm:mainPA} as a corollary of the following theorem.

\begin{theorem}
Suppose $f\colon \N \to \{0, 1\}$ is a computable partial function with no computable completion. Then every set $A\subseteq \N$ of positive lower density contains an infinite subset which does not compute a completion of $f$.
\end{theorem}

\noindent We show in Theorem \ref{thm:PA-counterexample} that the assumption that $f$ is computable is neccesary.

Our proof of this theorem follows the same general strategy as the proof of our main theorem, with a few notable changes. First, we replace Seetapun's theorem with Liu's theorem (as discussed in the introduction) and also make some more-or-less standard changes to the proof which are relevant for avoiding PA degrees (these changes are mostly taken from the work of Liu~\cite{liu2012} and of Monin and Patey~\cite{monin2021srt}). Second, and more interestingly, we can no longer rely on the cone avoiding basis theorem (the point is that there is no such thing as a ``PA degree avoiding basis theorem'' for somewhat obvious reasons). Thus we are forced to replace the argument using the cone avoiding basis theorem with a somewhat more elaborate argument. Third, we need an extra combinatorial fact about intersections of dense sets. This fact is not hard to prove, but it is not as simple as the combinatorics which appeared in the proof of our main theorem.

\subsection{Liu's theorem and Liu's lemma}

To prove the theorem above, we will use two results due to Liu. First, Liu's theorem, which, as we discussed in the introduction, is the analogue of Seetapun's theorem for this setting (and will play an analogous role in the proof). We stated a version of this theorem as Theorem~\ref{thm:liu-thm}, but we will use the following more general version.

\begin{theorem}[Liu's Theorem, \cite{liu2012}]
Suppose $f \colon \N \to \{0, 1\}$ is a computable partial function with no computable completion. Then for every set $A \subseteq \N$, there is an infinite set $B$ which is a subset of either $A$ or the complement of $A$ such that $B$ does not compute a completion of $f$.
\end{theorem}

As with Seetapun's theorem, it is easy to use this theorem to prove the following corollary.

\begin{corollary}
Suppose $f \colon \N \to \{0, 1\}$ is a computable partial function and $A \subseteq \N$ is an infinite set which does not compute a completion of $f$. Then for every finite partition $B_1,\ldots, B_n$ of $A$, at least one of the $B_i$'s has an infinite subset which does not compute a completion of $f$.
\end{corollary}

In the course of proving Theorem~\ref{thm:liu-thm}, Liu implicitly used the following lemma (see Lemma 6.6 of \cite{liu2012}). It was also used in a slightly different form by Monin and Patey in~\cite{monin2021srt} (see Lemmas 2.13 and 3.12).

\begin{definition}
A \term{valuation} is a finite partial function $p \colon \N \to \{0, 1\}$ represented as a lookup table (rather than, say, a code for a Turing machine which computes $p$). 
\end{definition}

\begin{definition}
Given a valuation $p$ and a partial function $f \colon \N \to \{0, 1\}$, $p$ is \term{$f$-correct} if $p \subseteq f$---in other words if for all $n$ in the domain of $p$, $n$ is also in the domain of $f$ and $p(n) = f(n)$.
\end{definition}

\begin{lemma}[Liu \cite{liu2012}]
If $W$ is a c.e.\ set of valuations and $f$ is a computable partial function with no total completion then either $W$ contains some $f$-correct valuation or for every $k$, there are at least $k$ many incompatible valuations outside of $W$.
\end{lemma}

\begin{proof}
Suppose $W$ does not contain any $f$-correct valuations and fix a number $k$. For any finite set $s \subseteq \N$ and valuation $p$, say that $p$ is \term{$f$-correct mod $s$} if $p$ is $f$-correct when we ignore the elements of $s$---in other words, for all $n \in \dom(p) \setminus s$, $n$ is in the domain of $f$ and $p(n) = f(n)$. We will show that there are numbers $n_1,\ldots,n_k$ such that no element of $W$ is $f$-correct mod $\{n_1, \ldots, n_k\}$. In particular, this implies that $W$ does not contain any element with domain contained in $\{n_1,\ldots,n_k\}$. Since it is easy to see that there are at least $k$ such incompatible functions (in fact, at least $2^k$), this is sufficient to prove the lemma.

We will construct $n_1,\ldots,n_k$ by induction. Suppose we have constructed $n_1,\ldots,n_i$ and want to find $n_{i + 1}$. In other words, we know $W$ contains no element which is $f$-correct mod $\{n_1,\ldots,n_i\}$ and we want to find some $n \notin \{n_1,\ldots,n_i\}$ such that no element is $f$-correct mod $\{n_1,\ldots,n_i, n\}$. 

Suppose for contradiction that no such $n$ exists. Then for every $n \notin \{n_1,\ldots,n_i\}$, $W$ contains some element $p$ which is $f$-correct mod $\{n_1,\ldots, n_i, n\}$. At the same time, $p$ cannot be $f$-correct mod $\{n_1,\ldots,n_i\}$. By unrolling the definitions, it is possible to see that this implies that $n$ is in the domain of $p$ and either not in the domain of $f$ or $p(n) \neq f(n)$. The key point is that in either case, $1 - p(n)$ does not disagree with $f(n)$.

But this gives us a method to compute a completion $g$ of $f$: For each $n \in \{n_1,\ldots,n_i\}$, hard-code some appropriate value for $g(n)$. For each $n \notin \{n_1,\ldots,n_i\}$, search for some $p \in W$ which is $f$-correct mod $\{n_1,\ldots,n_i, n\}$ (note that there is a computable enumeration of $f$-correct valuations). For the first such $p$ that is found, set $g(n) = 1 - p(n)$. Since $f$ has no computable completions, this gives us the desired contradiction.
\end{proof}

\subsection{A combinatorial lemma}

As we referred to above, we will need the following combinatorial lemma, which says that when you have enough dense sets, two of them must have fairly dense intersection. The proof is a standard counting argument, which we phrase as a probabilistic proof based on estimating the variance of a certain random variable.

\begin{lemma}
For every $\delta > 0$, there is some $k$ with the following property: For any $n > 0$ and any $k$ subsets $A_1,\ldots,A_k \subseteq [0, n)$, all of which have size at least $\delta n$, there is some pair of indices $i \neq j$ such that $\frac{A_i \cap A_j}{n} \geq \delta^2/2$.
\end{lemma}

\begin{proof}
We will begin by thinking of $k$ as a variable and fixing a number $n$ and sets $A_1,\ldots, A_k \subseteq [0, n)$. We will then find a lower bound on the maximum of $\frac{A_i\cap A_j}{n}$ over all $i \neq j$. At the end, we will check that if $k$ is large enough then this lower bound is at least $\delta^2/2$ (and does not depend on $n$).

So: fix $n$ and sets $A_1, \ldots, A_k$. Consider choosing $x$ from $[0, n)$ uniformly at random and counting the number of sets $A_i$ which contain $x$. Let $X$ be the random variable denoting this count. For each $i$, let $1_{A_i}$ be the random variable indicating whether $x$ is in $A_i$ or not ($1_{A_i} = 1$ if $x \in A_i$ and $0$ otherwise). Note that we have
\[
    X = \sum_{i \leq k} 1_{A_i}.
\]  

The key to the proof is to estimate the variance of $X$. First, by a standard calculation we have
\begin{align*}
\Var[X] &= \E[X^2] - \E[X]^2 \\
        &= \E\Bigl[\bigl(\sum_{i \leq k}1_{A_i}\bigr)^2\Bigr] - E\Bigl[\sum_{i \leq k}1_{A_i}\Bigr]^2\\
        % &= \E[\sum_{i, j}1_{A_i}1_{A_j}] - E[\sum_{i}1_{A_i}]^2\\
        &= \sum_{i, j \leq k}\E\left[1_{A_i}1_{A_j}\right] - \bigl(\sum_{i \leq k} \E[1_{A_i}]\bigr)^2\\
        &= \sum_{i, j} \frac{|A_i\cap A_j|}{n} - \bigl(\sum_{i}\frac{|A_i|}{n}\bigr)^2\\
        &= \sum_{i \neq j}\frac{|A_i\cap A_j|}{n} + \sum_i \frac{|A_i|}{n} - \bigl(\sum_i \frac{|A_i|}{n}\bigr)^2.
\end{align*}
Using the fact that $\Var[X] \geq 0$ and letting $y = \sum_i \frac{|A_i|}{n}$, we have
\[ 
    \sum_{i \neq j} \frac{|A_i \cap A_j|}{n} \geq y^2 - y.
\]
Since each $A_i$ has size at least $\delta n$, $y \geq \delta k$. Let's assume $k > 1/\delta$ and thus $y > 1$. Since $y^2 - y$ is monotonic in $y$ for $y > 1$, we have 
\[
    \sum_{i \neq j} \frac{|A_i \cap A_j|}{n} \geq y^2 - y \geq (\delta k)^2 - \delta k.
\]
We can now calculate a lower bound $\epsilon$ on the maximum value of $\frac{|A_i\cap A_j|}{n}$ over all $i \neq j$. For we have
\[
    \sum_{i \neq j} \frac{|A_i \cap A_j|}{n} \leq \epsilon k (k-1) \leq \epsilon k^2
\]
and thus
\[
    \epsilon k^2 \geq (\delta k)^2 - \delta k.
\]
Simple algebraic manipulation then yields $\epsilon \geq \delta^2 - \delta/k$.

We are now in position to finish the proof. Recall that we wanted $k$ large enough that $\epsilon \geq \delta^2/2$. The lower bound above shows that it's enough to take $k \geq 2/\delta$. Since this value does not depend on $n$, we are done.
\end{proof}

\subsection{The proof}

For the rest of this section, fix a computable partial function $f$ with no computable completion. Define a modified version of density Mathias forcing where conditions are Mathias conditions $(s, A)$ such that there is some infinite set $D$ which does not compute a completion of $f$ and so that $A$ is dense along $D$. It suffices to prove the following key lemma; the rest of the proof can be completed as in the proof of Theorem~\ref{thm:main}.

\begin{lemma}
For any condition $(s, A)$ and Turing functional $\Phi$, there is some condition $(s', A') \leq (s, A)$ such that for all sets $B$ compatible with $(s', A')$, $\Phi(B)$ is not a completion of $f$.
\end{lemma}

\begin{proof}
Fix a set $D$ and number $\delta > 0$ such that $A$ is $\delta$-dense along $D$ and $D$ does not compute a completion of $f$.

Given a set $B$ and a valuation $p$, say that $p$ is \term{consistent} with $B$ if for all finite sets $t \subseteq B$, $\Phi(s\cup t)$ and $p$ are compatible (i.e.\ for every $n$ either at least one of $\Phi(s\cup t, n)$ and $p(n)$ is not defined or they are both defined and are equal). Say that a sequence of sets $B_1,\ldots,B_l$ is \term{lawful} if it satisfies the following criteria:
\begin{enumerate}
    \item The $B_i$'s are all disjoint.
    \item For each $B_i$, there are incompatible valuations $p$ and $q$ such that $B_i$ is consistent with both $p$ and $q$. Note that this implies that for all $C$ compatible with $(s, B_i)$, $\Phi(C)$ must be compatible with both $p$ and $q$ and hence cannot be total.
    \item There is an infinite set $E$ such that each $B_i$ is $\delta^2/8$-dense along $E$, $A \setminus (B_1\cup \ldots \cup B_l)$ is $\delta/2$-dense along $E$ and $E$ does not compute a completion of $f$.\footnote{For avoiding a cone in our main theorem, we were also able to have that $B_1\oplus \cdots \oplus B_l \oplus E$ avoids the cone. This was because of the cone avoiding basis theorem. Here, we are not able to ask that $B_1 \oplus \cdots \oplus B_l \oplus E$ does not compute a completion of $f$. Before, we used $B_1,\ldots,B_l$ in the definition of $\CC_{n,b}$, so that it was a $\Pi^0_1$ class relative to $B_1 \oplus \cdots \oplus B_k \oplus E$. Now, when defining the analogous class, it must be a $\Pi^0_1$ class relative only to $E$ alone.
    
    One consequence of this is that while previously we constructed our sequences $B_1,\ldots,B_l$ by extension, here, the sequence of length $l+1$ will not necessarily extend the sequence of length $l$.}
\end{enumerate}
We will prove that if there is a lawful sequence of length $l$ then either there is a condition $(s', A')$ with the property we want or there is a lawful sequence of length $l + 1$. Also, since the $B_i$ are disjoint and $\delta^2/8$-dense along $E$, there is no lawful sequence of length strictly greater than $8/\delta^2$, which suffices to finish the proof.

\medskip\noindent\textbf{Finding lawful sequences.} Suppose that $B_1, \ldots, B_l$ is a lawful sequence, as witnessed by $E$ and $(p_1, q_1), (p_2, q_2), \ldots, (p_l, q_l)$. We want to either find $(s', A') \leq (s, A)$ satisfying the conclusion of the lemma or a lawful sequence of length $l + 1$.\footnote{A slightly subtle point here is that we allow $l = 0$---that is, we allow the starting sequence to be empty. As in the proof of our main theorem, it is possible to check that the whole proof still works even in this case.}

For every valuation $r$, define $\CC_r$ to be the set of tuples $(X_1, \ldots, X_l, Y)$ such that
\begin{enumerate}
    \item $X_1, \ldots, X_l, Y$ are all disjoint.
    \item Each $X_i$ is $\delta^2/8$-dense along $E$.
    \item $Y$ is $\delta/2$-dense along $E$.
    \item For each $i$, $X_i$ is consistent with both $p_i$ and $q_i$.
    \item $Y$ is consistent with $r$.
\end{enumerate}
Note that the sets $\CC_r$ are uniformly $\Pi^0_1(E)$. Let $W$ be the set of valuations $r$ such that $\CC_r$ is empty. Note that $W$ is c.e.\ relative to $E$. Since $E$ does not compute a completion of $f$, we can apply Liu's Lemma relative to $E$ to get that one of the following holds.
\begin{enumerate}
    \item There is some $r \in W$ which is $f$-correct.
    \item For every $k$, there are at least $k$ incompatible valuations outside of $W$.
\end{enumerate}

\medskip\noindent\textbf{Case 1: an $f$-correct valuation in $W$.} Suppose that some $r \in W$ is $f$-correct. Thus $\CC_r$ is empty. In particular, $(B_1, \ldots, B_l, A\setminus (B_1\cup\ldots\cup B_l))$ is not in $\CC_r$. By our choice of $B_1,\ldots,B_l$ and $E$, this can only be because $r$ is not consistent with $A\setminus (B_1\cup \ldots \cup B_l)$. In other words, there is some finite set $t \subset A$ and some $n$ such that $\Phi(s\cup t, n)$ and $r(n)$ are both defined and not equal. Moreover, since $r$ is $f$-correct, this implies that $\Phi(s\cup t, n)$ and $f(n)$ are both defined and not equal. In this case we are done because we can set $(s', A') = (s \cup t, A \setminus [\max(s\cup t)])$.

\medskip\noindent\textbf{Case 2: many incompatible valuations outside $W$.} Suppose that there are $k$ many incompatible valuations $r_1,\ldots,r_k$ outside $W$, where $k$ is chosen as in the combinatorial lemma (but with $\delta/2$ rather than $\delta$). Thus we can find sequences 
\begin{align*}
(B_1^1, \ldots, &B_l^1, C_1) \in \CC_{r_1}\\
(B_1^2, \ldots, &B_l^2, C_2) \in \CC_{r_2}\\
&\vdots\\
(B_1^k, \ldots, &B_l^k, C_k) \in \CC_{r_k}.
\end{align*}
Since each $C_i$ is $\delta/2$-dense along $E$, it follows from the combinatorial lemma (and our choice of $k$) that for each $n \in E$, there are $i \neq j$ such that $C_i \cap C_j$ is $\delta^2/8$-dense at $n$. We can now apply the same trick we used in the proof of our main theorem (but using Liu's theorem instead of Seetapun's theorem) to get an infinite set $E_1 \subseteq E$ and a pair $i \neq j$ such that $C_i \cap C_j$ is $\delta^2/8$-dense along $E_1$ and $E_1$ does not compute a completion of $f$.

Now let $B_1', B_2', \ldots, B_l'$ denote $B_1^i, B_2^i, \ldots, B_l^i$ and let $B_{l + 1}'$ denote $C_i \cap C_j$. We have the following facts about the sequence $B_1', \ldots, B_{l + 1}'$:
\begin{enumerate}
    \item They are all disjoint. This follows from the facts that $B_1^i,\ldots,B_l^i$ are all disjoint from each other and from $C_i$ and that $B_{l + 1}' \subseteq C_i$.
    \item For each $B_h'$, there are incompatible valuations $p$ and $q$ such that $B_h'$ is consistent with both $p$ and $q$. For $B_1',\ldots,B_l'$ we can just take $p_h$ and $q_h$. For $B_{l + 1}'$ we can take $r_i$ and $r_j$.
    \item $E_1$ is infinite, does not compute a completion of $f$ and each $B_h'$ is $\delta^2/8$-dense along $E_1$. For $B_1',\ldots,B_l'$, this last part is because $E_1 \subseteq E$. For $B_{l + 1}'$, this follows from our choice of $E_1$.
\end{enumerate}
The point is that $B_1',\ldots,B_{l + 1}'$ is very close to a lawful sequence of length $l + 1$. All that's missing is that $A \setminus (B_1'\cup \ldots \cup B_{l + 1}')$ be $\delta/2$-dense along $E_1$. However, this is easily fixed.

Again using the same trick as in the proof of our main theorem, we can get an infinite set $E_2 \subseteq E_1$ which does not compute a completion of $f$ and such that either
\begin{itemize}
    \item $A\cap (B_1'\cup\ldots\cup B_{l + 1}')$ is $\delta/2$-dense along $E_2$
    \item or $A\setminus (B_1'\cup\ldots\cup B_{l + 1}')$ is $\delta/2$-dense along $E_2$.
\end{itemize}
In the first case, we can use the trick once more to find some infinite $E_3 \subseteq E_2$ which does not compute a completion of $f$ and some $h \leq l + 1$ such that $A \cap B_h'$ is dense along $E_3$ and so we can finish by taking $(s', A') = (s, A\cap B_h')$. In the latter case, we have found a lawful sequence of length $l + 1$.
\end{proof}

\subsection{A false generalization}

There is a common generalization of both Theorem \ref{thm:main} and Theorem \ref{thm:mainPA} which at one point the authors thought might be true. Namely, that given \textit{any} partial function $f\colon \N \to \{0, 1\}$ with no computable completion, and any set $A \subseteq \N$ of positive lower density, there is an infinite subset of $A$ that does not compute any completion of $f$. (Theorem \ref{thm:main} is the case when $f$ is already total, and Theorem \ref{thm:mainPA} is the case when $f$ is computable.) However, this is false.

\begin{theorem}\label{thm:PA-counterexample}
    There is a partial function $f \colon \mathbb{N} \to \{0,1\}$ and a set $A \subseteq \mathbb{N}$ of positive lower density such that any infinite subset $B \subseteq A$ computes a completion of $f$.
\end{theorem}
\begin{proof}
We will define a partial function $f$ while simulateneously defining $x_0,x_1,x_2,\ldots$.  We will make use of a computable function $\pi \colon \N \to \N$ such that $\pi^{-1}(n)$ is infinite for each $n$. For each $x_i$, we will have $\pi(x_i) = i$. The idea is that while the construction will be non-computable, the function $\pi(x)$ will allow us to computably recover from $x$ the function $\Phi_{\pi(x)}$ that $x$ would have been used to diagonalize against if $x$ was used to diagonalize at all.

Begin with $x_0 = 0$. Suppose that we have defined $x_i$. We have two cases:
\begin{enumerate}
    \item If $\Phi_i(x_i) \downarrow$, let $s_i$ be the number of stages it takes to converge, and define $f(x_i) = 1- \Phi_i(x_i)$. Let $x_{i+1} > s_i$ be such that $\pi(x_{i+1}) = i+1$.
    \item Otherwise, if $\Phi_i(x_i) \uparrow$, leave $f(x_i)$ undefined, and let $x_{i+1} > x_i$ be such that $\pi(x_{i+1}) = i+1$.
\end{enumerate}
For any $x$ outside of $\{x_0, x_1, \ldots\}$, $f(x)$ is undefined.
% The domain of $f$ is $\{x_0,x_1,\ldots\}$.

Now define $A \subseteq 2^{< \omega}$ as follows. Put $\sigma \in A$ if whenever $\Phi_i(x_i) \downarrow$ but $\Phi_{i,|\sigma|}(x_i) \uparrow$, $\sigma(x_i) = f(x_i)$. We claim that $A$ has lower density at least $1/4$. For this, it is important that in identifying $2^{< \omega}$ with $\mathbb{N}$, we list out the finite binary strings in order of increasing length. Thus it suffices to show that for each $n$, there is some $x < n$ and $t \in \{0,1\}$ such that $A \cap 2^n \supseteq \{ \sigma \in 2^n \mid \sigma(x) = t\}$. If not, there would be $i < j$ with $x_i < x_j < n$ such that
\begin{enumerate}
	\item $\Phi_i(x_i) \downarrow$ but $\Phi_{i,n}(x_i) \uparrow$ and
	\item $\Phi_j(x_j) \downarrow$ but $\Phi_{j,n}(x_j) \uparrow$.
\end{enumerate}
From (1), it follows that $n < s_i$, but by construction $x_j > s_i$. This contradicts the choice of $x_j < n$.

Now suppose that $B$ is an infinite subset of $A$. We will compute from $B$ a completion $g$ of $f$. For each $x$, look for some $\sigma \in B$ with $\sigma(x) \downarrow$. Check whether $\Phi_{\pi(x),|\sigma|}(x) \downarrow$. If so, define $g(x) = 1- \Phi_{\pi(x),|\sigma|}(x)$. Otherwise, define $g(x) = \sigma(x)$. It is easy to see that $g$ is total.

To see that $g$ is a completion of $f$, it suffices to check that for all $i$, $g(x_i) = f(x_i)$. Let $\sigma \in B$ be the string used to define $g(x_i)$. If $\Phi_{i,|\sigma|}(x_i) \downarrow$, then we defined both $g(x_i)$ and $f(x_i)$ to be $1-\Phi_{i}(x_i)$. Otherwise, if $\Phi_{i,|\sigma|}(x_i) \uparrow$, then we defined $g(x_i) = \sigma(x_i)$, and if $f(x_i)$ is defined then $\Phi_i(x_i) \downarrow$ and so, as $\sigma \in A$, $\sigma(x_i) = f(x_i)$.
\end{proof}

\section{Preserving Kolmogorov complexity}
\label{sec:kolmogorov}

Recall from the introduction that for any string $\sigma$ and any collection of sets $\F \subseteq \powerset(\N)$, we define
\[
  C(\sigma \mid \F) = \max_{X \in \F}C^X(\sigma).
\]
Also recall that $C(\sigma) - C(\sigma \mid [A]^\omega)$ can be thought of as the number of bits of information about $\sigma$ coded into all infinite subsets of $A$. The goal of this section is to investigate this quantity in the case where $A$ is a set of positive lower density.

We noted in the introduction that it is possible to encode about $\log(1/\delta)$ bits of information into all infinite subsets of a set of lower density $\delta$. For example, if $\delta = 1/2^n$ then for any binary string $\sigma$ of length $n$, we can encode $\sigma$ into all infinite subsets of a set $A$ of lower density $1/2^n$ by letting $A$ consist of all numbers $m$ such that the last (i.e.\ least significant) bits of the binary expansion of $m$ match $\sigma$.

However, this is not the only way to encode information into all infinite subsets of a dense set. Roughly speaking, for any fixed number $N$, it is possible to encode ``an arbitrary integer greater than $N$.'' We will now make this precise.

\begin{definition}
For any string $\sigma \in 2^{<\omega}$ and number $N \in \N$, define
\[
  C(\sigma \mid \; \geq N) = \max_{n \geq N}C(\sigma \mid n).
\]
\end{definition}

\begin{proposition}
For any string $\sigma$ and number $N$, there is some set $A$ of lower density 1 such that
\[
  C(\sigma \mid [A]^{\omega}) \leq C(\sigma \mid\; \geq N) + O(1)
\]
where the constant hidden by the $O(1)$ does not depend on $\sigma$ or $N$.
\end{proposition}

\begin{proof}
Let $A = \{n \mid n \geq N\}$. If $B$ is an infinite subset of $A$ then all elements of $B$ are greater than or equal to $N$. In other words, from any infinite subset of $A$ we can uniformly extract a number greater than or equal to $N$. Thus $C(\sigma \mid [A]^\omega)$ is at most $C(\sigma \mid\; \geq N) + O(1)$, where the extra constant represents the complexity of the extraction procedure.
\end{proof}

It may not be immediately apparent exactly how much information can be gained from knowing ``an arbitrary integer greater than $N$'' (i.e.\ how large $C(\sigma) - C(\sigma \mid\; \geq N)$ can be). A result of Vereshchagin~\cite{vereshchagin2002kolmogorov} gives a precise characterization.

\begin{theorem}[Vereshchagin]
For any string $\sigma$,
\[
  C^{0'}(\sigma) = \min_{N}C(\sigma \mid\; \geq N) \pm O(1).
\]
\end{theorem}

Together with our proposition above, this shows that for any string $\sigma$, it is possible for all infinite subsets of a dense set $A$ to lower the complexity of $\sigma$ to its $0'$ complexity. This can be combined with the first coding method we mentioned to give a result claimed in the introduction.

\upperbound*

\begin{proof}
Let $k = C^{0'}(\sigma)$ and let $\rho$ be a string witnessing this fact---i.e.\ a string such that $|\rho| = k$ and for the universal oracle machine $U$, $U^{0'}(\rho) = \sigma$. It is a standard fact about Kolmogorov complexity that there is some $N$ such that for all $n \geq N$, $C(\sigma \mid \rho, n) = O(1)$ (this fact is more or less the easy direction of Vereshchagin's theorem mentioned above).

Next, let $m \in \N$ be such that $2^{-m - 1} < \delta \leq 2^{-m}$ and define $A$ by
\begin{align*}
  A = \{n \geq N \mid &\text{ the binary representation of $n$ mod $2^m$}\\
  &\qquad\qquad\text{is equal to the first $m$ bits of $\rho$}\}.
\end{align*}
If the length of $\rho$ is less than $m$ then first extend $\rho$ to length $m$ by adding a $1$ followed by a run of $0$s to the end.

Note that the lower density of $A$ is $2^{-m}$, which is greater than or equal to $\delta$. Furthermore, from any infinite subset $B \subseteq A$, we can uniformly extract a number $n \geq N$. And if we know $m$ then we can also uniformly extract the first $m$ bits of $\rho$. Since $C(\sigma \mid \rho, n) = O(1)$, to reconstruct $\sigma$ from this information, we only need to know the last $k - m$ bits of $\rho$. In other words, if we are given an infinite subset $B \subseteq A$ then to reconstruct $\sigma$ we just need to know $m$ and the last $k - m$ bits of $\rho$. Thus we have
\[
  C(\sigma \mid [A]^\omega) \leq (k - m) + O(C(m)) \leq k - m + O(\log(m)).
\]
Since $\log(1/\delta) \geq m \geq \log(1/\delta) - 1$, this gives us
\[
  C(\sigma \mid [A]^\omega) \leq k - \log(1/\delta) + O(\log\log(1/\delta))
\]
as desired.
\end{proof}

In this section, we will show that this upper bound is essentially optimal. In particular, we will prove the following theorem.

\lowerbound*

%\begin{theorem}
%\label{thm:ds_kolmogorov}
%For any string $\sigma$ and set $A \subseteq \N$ of lower density at least $\delta \in (0, 1]$,
%\[
%  C(\sigma \mid [A]^\omega) \geq C^{0'}(\sigma) - \log(1/\delta) - O(\log\log(1/\delta)).
%\]
%where the constant hidden by the $O(\cdot)$ notation does not depend on $\sigma$ or $A$.
%\end{theorem}

It is also possible to prove a corresponding result for Seetapun's theorem. In particular, for any set $A$, define
\[
  \RT^1_2(A) = [A]^\omega \cup [\bar{A}]^\omega.
\]
The quantity $C(\sigma) - C(\sigma \mid \RT^1_2(A))$ can be viewed as the number of bits of information about $\sigma$ that can be encoded into all infinite subsets of $A$ and $\bar{A}$. Just as in the case of subsets of a dense set, it is always possible to lower the complexity of $\sigma$ to its $0'$ complexity.

\begin{proposition}
\label{prop:seetapun_kolmogorov}
For any string $\sigma$, there is some set $A \subseteq \N$ such that
\[
  C(\sigma \mid \RT^1_2(A)) \leq C^{0'}(\sigma) + O(1).
\]
\end{proposition}

\begin{proof}
Once again, by Vereshchagin's theorem it is enough to prove that for any $N$ there is some set $A$ such that $C(\sigma \mid \RT^1_2(A)) \leq C(\sigma \mid\; \geq N) + O(1)$. And for this, we can simply take $A = \{n \mid n \geq N\}$.
\end{proof}

Just like in the case of dense sets, it is possible to show that this is optimal (and unlike in the case of dense sets, the error term here is constant).

\begin{theorem}
\label{thm:seetapun_kolmogorov}
For any string $\sigma$ and set $A \subseteq \N$,
\[
  C(\sigma \mid \RT^1_2(A)) \geq C^{0'}(\sigma) - O(1)
\]
where the constant hidden by the $O(1)$ does not depend on $\sigma$ or $A$.
\end{theorem}

The remainder of this section is focused on proving Theorems~\ref{thm:ds_kolmogorov} and~\ref{thm:seetapun_kolmogorov}. As a warm-up, and to demonstrate our proof strategy, we will first prove a much easier version of Theorem~\ref{thm:seetapun_kolmogorov}. We will then prove the full theorem and, finally, adapt the techniques from the proof to prove Theorem~\ref{thm:ds_kolmogorov}.

\subsection{Kolmogorov complexity and Seetapun's theorem: easy version}

We will now prove a version of Theorem~\ref{thm:seetapun_kolmogorov} where the $0'$ oracle is replaced by an oracle for a complete $\Sigma^1_2$ set. In particular, let $W$ denote a complete $\Sigma^1_2$ set and fix a string $\sigma$ and set $A \subseteq \N$. We will prove that
\[
  C(\sigma \mid \RT^1_2(A)) \geq C^W(\sigma) - O(1).
\]
This may appear a bit absurd---$W$ is ridiculously powerful---but it helps demonstrate the basic strategy that we will use in our proofs of Theorems~\ref{thm:seetapun_kolmogorov} and~\ref{thm:ds_kolmogorov}.

Here's the core idea of the proof. Suppose $k$ is a number such that $C(\sigma \mid \RT^1_2(A)) < k$. We would like to show that $C^W(\sigma) \leq k + O(1)$. To do this, we will identify some property of $\sigma$ which is shared by at most $2^k$ other strings and which can be recognized using an oracle for $W$. One obvious special property of $\sigma$ is that there is a set $X$ such that $C(\sigma \mid \RT^1_2(X)) < k$. It is not hard to show that at most $2^k$ strings have this property and that $W$ is powerful enough to check which strings have it (in fact, we chose $W$ to make this part obvious).

% we actually just chose $W$ to be powerful enough to check which strings have this property. 

We will now give a more formal version of this argument. We will describe a program $E$ with the following properties.
\begin{itemize}
\item For each $k$, $E^W(k)$ enumerates a set of at most $2^k$ strings.
\item For any set $A$, string $\sigma$ and number $k$ such that $C(\sigma \mid \RT^1_2(A)) < k$, $\sigma$ is in the set enumerated by $E^W(k)$.
\end{itemize}
By general properties of Kolmogorov complexity, for any number $k$ and string $\tau$ enumerated by $E^W(k)$, $C^W(\tau) \leq k + O(1)$. In particular, if $C(\sigma \mid \RT^1_2(A)) < k$ then $C^W(\sigma) \leq k + O(1)$.

We will now describe $E$. First, say that a string $\tau$ is \term{good-for-$k$} if there is a set $X$ such that $C(\tau \mid \RT^1_2(X)) < k$. We claim that both of the following hold.
\begin{enumerate}
\item It is computable in $W$ to check if a string $\tau$ is good-for-$k$ (i.e.\ the set of pairs $(\tau, k)$ such that $\tau$ is good-for-$k$ is computable relative to $W$).
\item There are at most $2^k$ strings which are good-for-$k$.
\end{enumerate}
To see why the first claim holds, simply note that the statement ``$\tau$ is good-for-$k$'' is equivalent to the statement ``there is some $X$ such that for every $Y$, if $Y \subseteq X$ or $Y \subseteq \bar{X}$, then $C^Y(\tau) < k$,'' which is a $\Sigma^1_2$ formula of $\tau$ and $k$.

To see why the second claim holds, suppose that $\tau_1,\ldots,\tau_l$ are distinct good-for-$k$ strings, as witnessed by $X_1,\ldots,X_l$. Note that the collection of all Boolean combinations of the sets $X_1,\ldots, X_l$ forms a finite partition of $\N$. Thus one of these Boolean combinations is infinite. Let $B$ be this Boolean combination and note that for each $i \leq l$, $B$ is either a subset of $X_i$ or of $\bar{X_i}$ and, either way, is an element of $\RT^1_2(X_i)$. Therefore for each $i \leq l$, $C^B(\tau_i) < k$, which is impossible unless $l \leq 2^k$.

The program $E^W(k)$ works as follows. Given the input $k$, it goes through all strings in some fixed order and uses $W$ to check whether each one is good-for-$k$. Each time it finds a string which is good-for-$k$, it enumerates it. As we showed above, $E^W(k)$ will never enumerate more than $2^k$ strings. Also, for any set $A$, string $\sigma$ and number $k$ such that $C(\sigma \mid \RT^1_2(A)) < k$, it is obvious that $\sigma$ is good-for-$k$ and thus will be enumerated by $E^W(k)$. It follows that for such a $\sigma$ and $k$, $C^W(\sigma) \leq k + O(1)$.

\subsection{Kolmogorov complexity and Seetapun's theorem: hard version}

We will now prove Theorem~\ref{thm:seetapun_kolmogorov}, which states that for any string $\sigma$ and set $A \subseteq \N$, $C(\sigma \mid \RT^1_2(A)) \geq C^{0'}(\sigma) - O(1)$. Our strategy is the same as in the previous subsection---i.e.\ let's assume that $C(\sigma \mid \RT^1_2(A)) < k$ and identify a property of $\sigma$ which is shared by at most $2^k$ strings and which can be recognized using an oracle for $0'$.

In identifying such a property, there is a natural trade-off between how easy the property is to describe and the computational power required to check if a string has the property. In the previous subsection, we just needed a property which can be recognized using a complete $\Sigma^1_2$ set and this allowed us to use a very straightforward property of $\sigma$. Now, however, we want a property that can be recognized by $0'$, which forces us to use a property that has a more intricate description. The key definition is the following.

\begin{definition}
A finite set of strings $F$ is \term{$k$-safe} if there is some finite\footnote{We could take $n = 2^{|F|}$ but this is not needed.} partition $X_1,\ldots, X_n$ of $\N$ and some number $m$ such that for all $i \leq n$ and all finite subsets $s \subseteq X_i$,
\[
  |s| \geq m \implies |\{\tau \mid C^s(\tau) < k\} \cup F| \leq 2^k.
\]
\end{definition}

The idea is that $F$ is $k$-safe as witnessed by a partition $X_1,\ldots, X_n$ if for any $X_i$ and any infinite subset $B$ of $X_i$, we may safely assume that $B$ will assign all strings in $F$ complexity less than $k$. More specifically, it may not actually be the case that $B$ assigns a complexity less than $k$ to each string in $F$, but if we assume that it does then we will never see a contradiction of the form ``$B$ assigns complexity less than $k$ to too many strings.''

Essentially the key property of $\sigma$ is that it is a member of every maximal $k$-safe set. More precisely, we will show that any maximal $k$-safe set has size at most $2^k$ and contains $\sigma$ and that there is some maximal $k$-safe set which can be enumerated by $0'$, uniformly in $k$. The key facts are
\begin{enumerate}
\item Every $k$-safe set has size at most $2^k$.
\item If $F$ is $k$-safe then so is $F \cup \{\sigma\}$.
\item The collection of $k$-safe sets is c.e.\ relative to $0'$ (uniformly in $k$).
\end{enumerate}

Before proving these three facts, let's see how they can be used to finish the proof. Consider the following procedure for using $0'$ to enumerate a maximal $k$-safe set.
{\tt
\begin{enumerate}[1.]
\item Set $F = \0$
\item While true:
\item \quad Search for a string $\tau$ such that $F \cup \{\tau\}$ is $k$-safe
\item \quad If such a $\tau$ is found, enumerate $\tau$ and set $F = F\cup \{\tau\}$
\end{enumerate}}
\noindent It is easy to show by induction that at every step of the above algorithm, $F$ is $k$-safe. Since no $k$-safe set has size more than $2^k$, the above algorithm can add at most $2^k$ strings to $F$. Thus after some point, no new strings will be added to $F$. We claim that at this point, $\sigma$ must be an element of $F$. If not, then since $F\cup\{\sigma\}$ is $k$-safe, eventually the algorithm will discover this fact and add $\sigma$ to $F$, contradicting our assumption that no new elements are added to $F$. All this implies that the algorithm must eventually enumerate $\sigma$ and also that the algorithm will enumerate at most $2^k$ strings.

To finish, observe that the above procedure was uniform in $k$. In other words, there is a single program $E$ such that for each $l$, $E^{0'}(l)$ carries out the above algorithm with $k = l$. As in the previous subsection, for every number $l$ and string $\tau$ enumerated by $E^{0'}(l)$, $C^{0'}(\tau) \leq l + O(1)$. Since $\sigma$ is enumerated by $E^{0'}(k)$, this shows that $C^{0'}(\sigma) \leq k + O(1)$, as desired.

We will now prove each of the three facts above.

\begin{lemma}
\label{lemma:seetapun_kolmogorov1}
Every $k$-safe set has size at most $2^k$.
\end{lemma}

\begin{proof}
Suppose $F$ is a $k$-safe set. Let $X_1,\ldots, X_n$ be a partition witnessing that $F$ is $k$-safe. Since $X_1,\ldots,X_n$ is a finite partition of $\N$, some $X_i$ must be infinite. It is straightforward to check that since $X_i$ is infinite, we must have
\[
  |\{\tau \mid C^{X_i}(\tau) < k\} \cup F| \leq 2^k
\]
which implies that $|F| \leq 2^k$.
\end{proof}

\begin{lemma}
\label{lemma:seetapun_kolmogorov2}
If $F$ is a $k$-safe set then so is $F\cup \{\sigma\}$.
\end{lemma}

\begin{proof}
Suppose that $F$ is $k$-safe as witnessed by $X_1,\ldots,X_n$ and $m$. Recall that the string $\sigma$ and set $A$ have the property that for all infinite subsets $B$ of $A$ or $\bar{A}$, $C^B(\sigma) < k$. The main idea of this proof is that the fact that $F\cup\{\sigma\}$ is $k$-safe can be witnessed by the partition
\[
  X_1\cap A, X_1\cap \bar{A}, \ldots, X_n\cap A, X_n\cap\bar{A}.
\]
More precisely, let $m'$ be a number larger than both the size of any finite piece of this partition and $m$. We will show that $F\cup\{\sigma\}$ is $k$-safe as witnessed by $m'$ and $X_1\cap A, X_1\cap \bar{A},\ldots, X_n\cap A, X_n\cap \bar{A}$.

To prove this claim, suppose $s$ is a finite subset of some set in this partition and $|s| \geq m'$. Without loss of generality, let's assume that $s\subseteq X_1\cap A$. By our choice of $m'$ and the assumption that $|s| \geq m'$, $X_1\cap A$ must be infinite. Thus we can find some subset $B\subseteq X_1\cap A$ which agrees with $s$ below $\max(s)$ and which is infinite. It is straightforward to check the following facts
\begin{enumerate}
\item Since $B$ is an infinite subset of $X_1$, $|\{\tau \mid C^B(\tau) < k\} \cup F| \leq 2^k$.
\item Since $s$ is an initial segment of $B$, $\{\tau \mid C^s(\tau) < k\} \subseteq \{\tau \mid C^B(\tau) < k\}$.
\item Since $B$ is an infinite subset of $A$, $C^B(\sigma) < k$ and thus $\sigma \in \{\tau \mid C^B(\tau) < k\}$.
\end{enumerate}
Putting these together, we have
\begin{align*}
  \{\tau \mid C^s(\tau) < k\}\cup F \cup \{\sigma\} &\subseteq \{\tau \mid C^B(\tau) < k\} \cup F \cup \{\sigma\}\\
  &= \{\tau \mid C^B(\tau) < k\} \cup F
\end{align*}
and thus
\[
  |\{\tau \mid C^s(\tau) < k\}\cup F \cup \{\sigma\}| \leq |\{\tau \mid C^B(\tau) < k\}\cup F| \leq 2^k
\]
as desired.
\end{proof}

\begin{lemma}
\label{lemma:seetapun_kolmogorov3}
The collection of $k$-safe sets is c.e.\ relative to $0'$, uniformly in $k$.
\end{lemma}

\begin{proof}
Given a number $k$, a finite set of strings $F$, a number $m$ and sets $X_1,\ldots,X_n \subseteq \N$, the statement that $m$ and $X_1,\ldots,X_n$ witness that $F$ is $k$-safe is uniformly $\Pi^0_1$. In other words, there is some computable tree $T(k, F, m, n)$ such that $X_1\oplus \ldots \oplus X_n$ is a path through $T(k, F, m, n)$ if and only if $m$ and $X_1,\ldots,X_n$ witness that $F$ is $k$-safe and, furthermore, $T(k, F, m, n)$ is uniformly computable in $k$, $F$, $m$ and $n$.

By K\"onig's lemma, $T(k, F, m, n)$ has an infinite path if and only if $T(k, F, m, n)$ is itself infinite. Thus to check if $F$ is $k$-safe, it suffices to check if there are $m$ and $n$ such that $T(k, F, m, n)$ is infinite. This is a $\Sigma^0_2$ property (uniformly in $k$ and $F$) and thus c.e.\ relative to $0'$.
\end{proof}

\subsection{Kolmogorov complexity and dense sets}

We will now prove Theorem~\ref{thm:ds_kolmogorov}. Actually, we will instead prove the following apparently weaker version and then derive Theorem~\ref{thm:ds_kolmogorov} as a corollary.

\begin{theorem}
\label{thm:ds_kolmogorov2}
For any string $\sigma$ and set $A \subseteq \N$ which is $\delta$-dense for $\delta \in (0, 1]$,
\[
  C(\sigma \mid [A]^{\omega}) \geq C^{0'}(\sigma) - \log(1/\delta) - O(\log\log(1/\delta))
\]
where the constant hidden by the $O(\cdot)$ notation does not depend on $\sigma$ or $A$.
\end{theorem}

Note that this theorem differs from Theorem~\ref{thm:ds_kolmogorov} in that the set $A$ is required to be $\delta$-dense rather than to have lower density $\delta$ (essentially $A$ is required to be dense everywhere rather than just dense at all sufficiently large points).

We will begin by showing how to use this theorem to prove Theorem~\ref{thm:ds_kolmogorov}.

\begin{proof}[Proof of Theorem~\ref{thm:ds_kolmogorov} using Theorem~\ref{thm:ds_kolmogorov2}]
The basic point is that for any string $\sigma$ and set $A$ of lower density $\delta$, there is a set $\tilde{A}$ which is $\delta/4$-dense such that
\[
  C(\sigma \mid [\tilde{A}]^\omega) \leq C(\sigma \mid [A]^\omega) + O(1)
\]
and thus Theorem~\ref{thm:ds_kolmogorov2} implies that
\begin{align*}
  C(\sigma \mid [A]^\omega) &\geq C(\sigma \mid [\tilde{A}]^\omega) - O(1)\\
                            &\geq C^{0'}(\sigma) - \log(4/\delta) - O(\log\log(4/\delta))\\
                            &= C^{0'}(\sigma) - \log(1/\delta) - O(\log\log(1/\delta)).
\end{align*}

The set $\tilde{A}$ is simple to describe. Let $N$ be large enough that for all $n \geq N$, $A$ has density at least $\delta/2$ at $n$. Then define
\[
\tilde{A} = \{n \mid n \text{ is even and } n \leq 2N\} \cup \{2n + 1 \mid n \in A\}.
\]
In other words, $\tilde{A}$ contains a copy of $A$ on the odd numbers, together with enough even numbers to make sure $\tilde{A}$ is relatively dense everywhere.

To finish, note that any infinite subset of $\tilde{A}$ must contain an infinite number of odd numbers and thus can uniformly compute an infinite subset of $A$ and so we have
\[
  C(\sigma \mid [\tilde{A}]^\omega) \leq C(\sigma \mid [A]^\omega) + O(1)
\]
as desired.
\end{proof}

We will now prove Theorem~\ref{thm:ds_kolmogorov2}. The proof closely follows the proof from the previous subsection, but with a few additional difficulties. We will begin by modifying the definition of a $k$-safe set.

\begin{definition}
For any $k \in \N$ and rational $\delta \in (0, 1]$, a finite set of strings $F = \{\tau_1,\ldots,\tau_n\}$ is \term{$k$-safe at density $\delta$} if there is a number $m$ and sets $X_1,\ldots, X_n \subseteq \N$ of density at least $\delta$ such that for all $i_1 < i_2 < \ldots < i_l$ and all finite subsets $s \subseteq X_{i_1} \cap \ldots \cap X_{i_l}$,
\[
  |s| \geq m \implies |\{\tau \mid C^s(\tau) < k\} \cup \{\tau_{i_1}, \ldots, \tau_{i_l}\}| \leq 2^k.
\]
\end{definition}

\noindent As in the previous subsection, we have the following claims.
\begin{enumerate}
\item If $F$ is $k$-safe at density $\delta$ then $|F| \leq 2^k/\delta$.
\item Suppose we have a set $A$ and string $\sigma$ such that $A$ is $\delta$ dense and $C(\sigma \mid [A]^{\omega}) < k$. If $F$ is $k$-safe at density $\delta$ then so is $F \cup \{\sigma\}$.
\item The collection of sets which are $k$-safe at density $\delta$ is c.e.\ relative to $0'$, uniformly in $k$ and $\delta$.
\end{enumerate}

Before proving these claims, let's see how to use them to finish the proof. In a similar fashion to the previous subsection, we can find a program $E$ such that for any $k$ and $n$, $E^{0'}(k, n)$ enumerates a maximal set which is $k$-safe at density $2^{-n}$. Furthermore, $E^{0'}(k, n)$ enumerates at most $2^k\cdot 2^n$ strings. By standard facts about Kolmogorov complexity, this implies that for any $\tau$ which is enumerated by $E^{0'}(k, n)$, $C^{0'}(\tau \mid n) \leq k + n + O(1)$. Thus for such a $\tau$ we have
\[
  C^{0'}(\tau) \leq C^{0'}(\tau \mid n) + O(C^{0'}(n)) \leq k + n + O(C^{0'}(n))
\]
and so
\[C^{0'}(\tau) \leq k + n + O(\log(n)). 
\]
Now fix a string $\sigma$ and a set $A$ of density at least $\delta \in (0, 1/2]$. Let $k = C(\sigma \mid [A]^{\omega}) + 1$ and $n$ be such that $2^{-n} < \delta \leq 2^{-n + 1}$. Then by the facts above, we have that $\sigma$ is enumerated by $E^{0'}(k, n)$ and thus
\[
  C^{0'}(\sigma) \leq k + n + O(\log(n)) = k + \log(1/\delta) + O(\log\log(1/\delta)).
\]
Note that in the calculation above we used that $n - 1 \leq \log(1/\delta) < n$.

We will now prove the claims above.

\begin{lemma}
Suppose $F$ is $k$-safe at density $\delta$. Then $|F| \leq 2^k/\delta$.
\end{lemma}

\begin{proof}
Suppose $F = \{\tau_1,\ldots,\tau_n\}$ and that the sets $X_1,\ldots,X_n$ witness that $F$ is $k$-safe at density $\delta$.

We first claim that there is some collection of at least $\delta n$ many $X_i$'s whose intersection is infinite. To see why, suppose for contradiction that each such collection is finite. Then there is some number $N$ large enough that every $x \geq N$ is in fewer than $\delta n$ of the $X_i$'s. Fix some $N'$ much larger than $N$ and consider a random point $x$ from the interval $[N, N')$. We will derive a contradiction by calculating the expected number of $X_i$'s which contain $x$.

First consider a single fixed $X_i$. Since $X_i$ is $\delta$ dense at $N'$, $X_i$ has at least $\delta N'$ elements less than $N'$ and thus has at least $\delta N' - N$ elements in the interval $[N, N')$. So the probability that $x$ is in $X_i$ is at least
\[
  \frac{\delta N' - N}{N'} = \delta - \frac{N}{N'}.
\]
By linearity of expectation, this implies that the expected number of $X_i$'s which contain $x$ is at least
\[
  \delta n - \frac{Nn}{N'}.
\]
Thus there is some $x$ in the interval $[N, N')$ which is contained in at least $\delta n - \frac{Nn}{N'}$ many $X_i$'s. Since the number of $X_i$'s containing $x$ must be an integer, this $x$ must actually be contained in
\[
  \left\lceil \delta n - \frac{Nn}{N'} \right\rceil
\]
many $X_i$'s. If $N'$ is large enough then this is simply equal to $\ceil{\delta n}$. In other words, $x$ is contained in at least $\delta n$ many $X_i$'s, contradicting our assumption.

We have now shown that there is some collection of at least $\delta n$ many $X_i$'s whose intersection is infinite. Suppose for convenience that these $X_i$'s are simply $X_1,\ldots,X_l$ where $l = \ceil{\delta n}$ and let $B = X_1 \cap \ldots \cap X_l$. Note that by definition of ``$k$-safe at density $\delta$,'' we must have
\[
  |\{\tau \mid C^B(\tau) < k\}\cup\{\tau_1,\ldots,\tau_l\}| \leq 2^k,
\]
which is impossible unless $l \leq 2^k$. To summarize, we have $\delta n \leq 2^k$, which implies that $n \leq 2^k/\delta$, as desired.
\end{proof}

\begin{lemma}
Suppose $A$ is $\delta$ dense, $C(\sigma \mid [A]^{\omega}) < k$ and $F$ is $k$-safe at density $\delta$. Then $F\cup\{\sigma\}$ is also $k$-safe at density $\delta$.
\end{lemma}

\begin{proof}
The proof is essentially the same as the proof of Lemma~\ref{lemma:seetapun_kolmogorov2}. Namely, let $F = \{\tau_1,\ldots,\tau_n\}$ and fix sets $X_1,\ldots,X_n$ and a number $m$ witnessing that $F$ is $k$-safe at density $\delta$. Let $m' > m$ be large enough that for any intersection of finitely many of $X_1,\ldots,X_n, A$, if that intersection happens to be finite, then it has at most $m'$ elements. By copying the proof of Lemma~\ref{lemma:seetapun_kolmogorov2}, one can show that $F\cup\{\sigma\}$ is $k$-safe at density $\delta$, as witnessed by $X_1,\ldots,X_n, A$ and $m'$.
\end{proof}

\begin{lemma}
The collection of sets which are $k$-safe at density $\delta$ is c.e.\ relative to $0'$, uniformly in $k$ and $\delta$.
\end{lemma}

\begin{proof}
The proof is essentially the same as the proof of Lemma~\ref{lemma:seetapun_kolmogorov3}. The key point is that the statement ``$\{\tau_1,\ldots,\tau_n\}$ is $k$-safe at density $\delta$, as witnessed by $X_1,\ldots,X_n$ and $m$'' is $\Pi^0_1$.
\end{proof}

\subsection{Open questions}

The presence of the $0'$ oracle in Theorems~\ref{thm:seetapun_kolmogorov} and~\ref{thm:ds_kolmogorov} is somewhat unsatisfying. In one sense it is necessary: Propositions~\ref{prop:seetapun_kolmogorov} and~\ref{prop:ds_kolmogorov} show that it cannot be removed. However, there is another sense in which it may be possible to strengthen the two theorems.

For the sake of concreteness, we will focus on a possible strengthening of Theorem~\ref{thm:seetapun_kolmogorov}. In that theorem, we showed that for any set $A$ and string $\sigma$, $C(\sigma \mid \RT^1_2(A)) \geq C^{0'}(\sigma) - O(1)$. By Vereshchagin's theorem, this is equivalent to showing that for each string $\sigma$, there is some number $N$ such that $C(\sigma \mid \RT^1_2(A)) \geq C(\sigma \mid \; \geq N) - O(1)$. However, this allows the possibility that even for a fixed set $A$, there is a different such $N$ for each $\sigma$. It seems natural to wonder whether this is necessary, or whether, for each set $A$, there is a single $N$ which works for all $\sigma$.

\begin{conjecture}
For any set $A \subseteq \N$, there is a number $N$ such that for any string $\sigma$,
\[
  C(\sigma \mid \RT^1_2(A)) \geq C(\sigma \mid\; \geq N) - O(1).
\]
\end{conjecture}

Note that the set $A$ we constructed to show that the bound in Theorem~\ref{thm:seetapun_kolmogorov} is tight is not a counterexample to this conjecture: for that set $A$ there is a single number $N$ such that for all strings $\sigma$, $C(\sigma \mid \RT^1_2(A)) = C(\sigma \mid\; \geq N) \pm O(1)$.

% On the other hand, for the set $A$ we constructed to show that the bound in Theorem~\ref{thm:seetapun_kolmogorov} is tight, there is a single number $N$ such that for all strings $\sigma$, $C(\sigma \mid \RT^1_2(A)) = C(\sigma \mid\; \geq N) \pm O(1)$. This suggests the following conjecture.

% Recall that the way in which we were able to lower the complexity of $\sigma$ to its $0'$ complexity in Proposition~\ref{prop:seetapun_kolmogorov} was by fixing a number $N$ large enough that $C(\sigma \mid\; \geq N) = C^{0'}(\sigma) \pm O(1)$ and then finding a set $A$ such that all infinite subsets of $A$ and $\bar{A}$ can be used to find a number larger than $N$. However, the set $A$ that we built was tailored to do this for one specific value of $N$. This suggests the following conjecture.

% Note that by Vereshchagin's theorem, this would immediately imply Theorem~\ref{thm:seetapun_kolmogorov}, but it is not at all obviously implied by that theorem.
% In other words, can Theorem~\ref{thm:seetapun_kolmogorov} be improved to yield a single $N$ that works for all $\sigma$ rather than a different $N$ for each $\sigma$. 

One could also ask for a similar strengthening of Theorem~\ref{thm:ds_kolmogorov}.

\begin{conjecture}
For any set $A \subseteq \N$ of lower density $\delta$, there is a number $N$ and a string $\tau$ of length at most $\log(1/\delta)$ such that for any string $\sigma$,
\[
  C(\sigma \mid [A]^\omega) \geq C(\sigma \mid\; \geq N, \tau) - O(\log\log(1/\delta)).
\]
\end{conjecture}

There is another conjecture, also related to Theorem~\ref{thm:main} and Kolmogorov complexity, but quite distinct from the conjectures above, that we would like to mention here. Namely, it might be possible to strengthen Theorem~\ref{thm:main} to say that given $A$ of positive lower density, and $X$ uncomputable, not only is there a subset of $A$ that does not compute $X$, but there is a subset of $A$ which does not help to compress initial segments of $X$. 

\begin{conjecture}
    Let $X$ be any set. For any set $A \subseteq \N$ of lower density $\delta$ there is an infinite subset $B \subseteq A$ such that for infinitely many $n$,
    \[ C^B(X\uh n) \geq C(X\uh n) - O(1).\]
\end{conjecture}
Here we take inspiration from a theorem due to Chaitin (see \cite{downey2010algorithmic}, Theorem 3.4.4) that $X$ is computable if and only if $C(X\uh n) \leq C(n) + O(1)$ for all $n$. If $X$ is uncomputable, there is a function $f\colon \N \to \N$ such that $\limsup_{n \to \infty} f(n) = \infty$ and for all $n$, $C(X\uh n) \geq C(n) + f(n) - O(1)$. If $B \subseteq A$ is an infinite set which preserves the complexity of infinitely many initial segments of $X$, then for infinitely many $n$,
\[
C^B(X\uh n) \geq C(X\uh n) - O(1) \geq C(n) + f(n) - O(1) \geq C^B(n) + f(n) - O(1),
\]
and hence $X$ is not computable relative to $B$, which is exactly the conclusion of Theorem~\ref{thm:main}.

Note that in this last conjecture we must only ask that $C^B(X\uh n) \geq C(X\uh n) - O(1)$ for infinitely many $n$, rather than for all $n$, as whenever $B \subseteq A$ is uncomputable it will given shorter descriptions for infinitely many $n$, and thus $B$ can further compress even some initial segments of the empty set $X = \varnothing$.

\section{The relationship between the main theorem and Seetapun's theorem}

It seems clear that Theorem~\ref{thm:main} and Seetapun's theorem are closely related. Both give limitations on coding information into all infinite subsets of a set and we used Seetapun's theorem in our proof of Theorem~\ref{thm:main}. However, the two theorems are even more closely related than this suggests. In particular, Theorem~\ref{thm:main} directly implies Seetapun's theorem.

To see why, let $X$ be uncomputable and $A \subseteq \N$ be arbitrary. Consider the following set $B \subseteq \N$ of lower density $1/2$:
\[
  B = \{2n \mid n \in A\} \cup \{2n + 1 \mid n \notin A\}.
\]
Given any subset $C \subseteq B$, we can read off a subset $C_{\text{even}}$ of $A$ by looking at the even elements of $C$ and a subset $C_{\text{odd}}$ of $\bar{A}$ by looking at the odd elements of $C$. Moreover, if $C$ is infinite then at least one of $C_{\text{even}}$ and $C_{\text{odd}}$ must be infinite as well. By Theorem~\ref{thm:main}, $B$ must have some infinite subset $C$ that does not compute $X$. But then either $C_{\text{even}}$ is an infinite subset of $A$ which does not compute $X$ or $C_{\text{odd}}$ is an infinite subset of $\bar{A}$ which does not compute $X$.

The argument above does not constitute a new proof of Seetapun's theorem because we used Seetapun's theorem in our proof of Theorem~\ref{thm:main}. However, it might make one wonder whether there is a similar direct proof of Theorem~\ref{thm:main} from Seetapun's theorem (in which case, the somewhat complicated proof of Theorem~\ref{thm:main} that we gave in Section~\ref{sec:main} would be pointless). In this section, we will show that this is not the case. However, we first have to make precise what sort of proof we are trying to rule out.

\subsection{Strong omniscient computable reducibility}

The proof of Seetapun's theorem from Theorem~\ref{thm:main} that we gave above can be seen as an example of a \term{strong computable reduction}, a notion closely related to Weihrauch reducibility and reverse math which was first introduced by Dzhafarov~\cite{dzhafarov2015cohesive}.

\begin{definition}
Given two partial functions $P, Q \colon \Cantor \to \powerset(\Cantor)$, $P$ is \term{strongly computably reducible to $Q$}, written $P \leq_{sc} Q$, if for any $x \in \dom(P)$,
\begin{enumerate}
\item there is some $\tilde{x} \leq_T x$ such that $\tilde{x} \in \dom(Q)$
\item and for any $\tilde{y} \in Q(\tilde{x})$, there is some $y \leq_T \tilde{y}$ such that $y \in P(x)$.
\end{enumerate}
\end{definition}

Intuitively, a partial function $P \colon \Cantor \to \powerset(\Cantor)$ can be seen as a problem: elements $x$ of the domain of $P$ are \term{instances} of the problem and elements $y$ of $P(x)$ are \term{solutions} to the instance $x$. Under this interpretation, a problem $P$ is strongly computably reducible to a problem $Q$ if any instance $x$ of the problem $P$ can compute an instance $\tilde{x}$ of $Q$ such that any solution to $\tilde{x}$ can compute a solution to the original problem $x$.

As we said, the proof of Seetapun's theorem from Theorem~\ref{thm:main} that we gave above can be seen as an example of a strong computable reduction. First, we can define problems corresponding to Seetapun's theorem and Theorem~\ref{thm:main}:
\begin{itemize}
\item $\RT^1_2$ is the problem in which an instance is a set $A \subseteq \N$ and a solution to that instance is an infinite subset of either $A$ or $\bar{A}$.
\item $\DS$ (which stands for ``subset of a dense set'') is the problem in which an instance is a set $A \subseteq \N$ of positive lower density and a solution is an infinite subset of $A$.
\end{itemize}
Next, we can translate the main idea of the proof into a strong computable reduction of $\RT^1_2$ to $\DS$:
\begin{enumerate}
\item Given an instance $A$ of $\RT^1_2$, we can compute the instance $B = \{2n \mid n \in A\} \cup \{2n + 1 \mid n \notin A\}$ of $\DS$.
\item Given a solution $C$ to the instance $B$ (i.e.\ an infinite subset $C \subseteq B$), we can compute a solution to the instance $A$ of $\RT^1_2$ (namely, either $C_{\text{even}}$ or $C_{\text{odd}}$).
\end{enumerate}
Furthermore, it is easy to check that if $P$ and $Q$ are problems such that $P \leq_{sc} Q$ and $X$ is a set such that there is some instance $P$, all of whose solutions compute $X$ then there is some instance of $Q$ all of whose solutions compute $X$. Thus the fact that $\RT^1_2 \leq_{sc} \DS$ plus Theorem~\ref{thm:main} implies Seetapun's theorem.

So to show that there is no proof of this sort of Theorem~\ref{thm:main} from Seetapun's theorem, we can show that there is no strong computable reduction of $\DS$ to $\RT^1_2$. In fact, we will prove a somewhat stronger statement: instead of strong computable reductions of $\DS$ to $\RT^1_2$, we will consider \term{strong omniscient computable reductions}.

Strong omniscient computable reducibility is a natural weakening of strong computable reducibility, introduced by Monin and Patey~\cite{monin2019pi}, in which the transformation of instances is not required to be computable (though the transformation of solutions still is).

\begin{definition}
Given two partial functions $P, Q \colon \Cantor \to \powerset(\Cantor)$, $P$ is \term{strongly omnisciently computably reducible to $Q$}, written $P \leq_{soc} Q$, if for any $x \in \dom(P)$,
\begin{enumerate}
\item there is some $\tilde{x}$ such that $\tilde{x} \in \dom(Q)$
\item and for any $\tilde{y} \in Q(\tilde{x})$, there is some $y \leq_T \tilde{y}$ such that $y \in P(x)$.
\end{enumerate}
Note that the only difference from strong computable reducibility is that here, $\tilde{x}$ is not required to be computable from $x$.
\end{definition}

It is clear that $P \leq_{sc} Q$ implies $P \leq_{soc} Q$, but not always vice-versa, and hence showing that $\DS \nleq_{soc} \RT^1_2$ is a stronger result than showing that $\DS \nleq_{sc} \RT^1_2$.

\subsection{Theorem~\ref{thm:main} is not reducible to Seetapun's theorem}

We will now prove that $\DS \nleq_{soc} \RT^1_2$. In our proof, we will construct infinite sets $A$ and $G$ such that $A$ has lower density at least $1/2$ and no infinite subset of $G$ computes any infinite subset of $A$.

To see why constructing such an $A$ and $G$ is sufficient to prove $\DS \nleq_{soc} \RT^1_2$, suppose for contradiction that such an $A$ and $G$ exist and that $\DS \leq_{soc} \RT^1_2$. Viewing $A$ as an instance of $\DS$, we get a set $B$ (i.e.\ an instance of $\RT^1_2$) such that every infinite subset of $B$ and $\bar{B}$ (i.e.\ every solution to $\RT^1_2(B)$) computes an infinite subset of $A$ (i.e.\ a solution to $\DS(A)$). However, since $G$ is infinite, at least one of $G\cap B$ and $G\cap \bar{B}$ is also infinite and does not compute any infinite subset of $A$.

In our construction of $A$ and $G$, we will use a corollary of the Galvin-Prikry theorem~\cite{galvin1973borel}.

\begin{theorem}[Galvin-Prikry theorem]
For any Borel set $\X \subseteq \powerset(\N)$, there is some infinite set $B \subseteq \N$ such that either all infinite subsets of $B$ are in $\X$ or no infinite subsets of $B$ are in $\X$. 
\end{theorem}

\begin{corollary}
For any finite set $k$, Borel coloring $c \colon \powerset(\N) \to k$ and infinite set $B$, there is some infinite set $B' \subseteq B$ such that all infinite subsets of $B'$ have the same color.
\end{corollary}

Before proving the existence of $A$ and $G$, we will try to motivate the construction and explain how the Galvin-Prikry theorem will be used. Our goal is to construct $A$ and $G$ such that for all Turing functionals $\Phi$ and all infinite subsets $C \subseteq G$, $\Phi(C)$ is not an infinite subset of $A$. Suppose that instead of having to handle all Turing functionals, we just had to handle a single Turing functional, $\Phi$. By the Galvin-Prikry theorem, there is some infinite set $B_0$ such that either
\begin{enumerate}
\item for all infinite subsets $C \subseteq B_0$, $\Phi(C, 0)\conv = 1$
\item or for all infinite subsets $C \subseteq B_0$, $\Phi(C, 0)\neq 1$ (i.e.\ $\Phi(C, 0)$ either diverges or converges to something besides $1$).
\end{enumerate}
In the first case, we are done: we can take $G = B_0$ and make sure $0 \notin A$ (e.g.\ take $A = \N\setminus \{0\}$). The second case is more difficult. Taking $G = B_0$ is not enough because it does not give us much control over what is computed by infinite subsets of $G$---it just ensures that if $C$ is an infinite subset of $G$ then $0$ is not an element of the set $\Phi(C)$.

However, requiring $G$ to be a subset of $B_0$ does seem like \emph{progress} towards our goal: if we could similarly ensure that for each $n \in \N$ and infinite set $C \subseteq G$, $n$ is not an element of $\Phi(C)$, then we would be done (because this would imply that for each such $C$, $\Phi(C)$ is either not total or computes the empty set).

This suggests the following strategy: form a sequence of sets
\[
  B_0 \supseteq B_1 \supseteq B_2 \supseteq \ldots
\]
where for each $n$, $B_n$ is chosen using the Galvin-Prikry theorem so that either
\begin{enumerate}
\item for all infinite subsets $C \subseteq B_n$, $\Phi(C, n)\conv = 1$
\item or for all infinite subsets $C \subseteq B_0$, $\Phi(C, n)\neq 1$.
\end{enumerate}
If the first case holds for some $n$, we can take $G = B_n$ and forbid $n$ from being included in $A$. If the second case holds for every $n$, then we would like to take $G = \bigcap_{n \in \N} B_n$ since, as we noted above, this ensures that for each infinite $C \subseteq G$, $\Phi(C)$ is either not total or computes the empty set. However, there is one problem with this: $\bigcap_{n \in \N}B_n$ could be finite, or even empty.

To deal with this problem, we will construct $G$ using a sequence of Mathias conditions, at each step restricting the reservoir in a manner similar to what we have just described, while also adding some elements to the stem to ensure $G$ is infinite. We will now give the details of the proof. Because of the necessity of using Mathias conditions, of dealing with all Turing functionals (rather than just one), and of ensuring that $A$ has high density, our proof is somewhat more elaborate than the sketch we have just given. However, the key ideas are the same.

\begin{theorem}
\label{thm:ds_nle_rt}
$\DS \nleq_{soc} \RT^1_2$.
\end{theorem}

\begin{proof}
As explained above, we will show how to construct infinite sets $A$ and $G$ such that $A$ has lower density at least $1/2$ and no infinite subset of $G$ computes an infinite subset of $A$. To construct $G$, we will build a sequence of Mathias conditions
\[
  (\0, \N) = (s_0, B_0) \geq (s_1, B_1) \geq (s_2, B_2) \geq \ldots
\]
and take $G = \bigcup_n s_n$. Along the way, we will forbid certain numbers from being included in $A$ and then take $A$ to be the set of all non-forbidden numbers. For convenience, we will ensure that each $s_n$ has size exactly $n$.

For each $e$, we must satisfy the following requirement:
\begin{quote}
\textbf{Requirement $\bm{e}$:} for every infinite $C \subseteq G$, either there is some $m$ such that $\Phi_e(C, m)\conv = 1$ and $m \notin A$ or for all but finitely many $m$, $\Phi_e(C, m) \neq 1$.
\end{quote}
To satisfy requirement $e$, we will take some action at each step $n > e$ of the construction. Roughly speaking, on step $n$ we will try to satisfy the requirement for all $m$ in the interval $[2^{e + n + 3}, 2^{e + n + 4})$.

More precisely, to satisfy requirement $e$, we will ensure that for each step $n > e$ and each $t \subseteq s_n$, either
\begin{enumerate}
\item there is some $m \in [2^{e + n + 3}, 2^{e + n + 4})$ such that for all infinite sets $C$ compatible with $(t, B_n)$, $\Phi(C, m)\conv = 1$
\item or for all $m \in [2^{e + n + 3}, 2^{e + n + 4})$ and infinite sets $C$ compatible with $(t, B_n)$, $\Phi(C, m) \neq 1$.
\end{enumerate}
If the first case holds, we will also pick one such $m$ to forbid from $A$.

\medskip\noindent\textbf{How to pick $\bm{s_{n + 1}}$ and $\bm{B_{n + 1}}$.}
For each $n$, form $s_{n + 1}$ by picking an arbitrary element of $B_n$ to add to $s_n$. Then pick $B_{n + 1}$ as follows.

For each $e < n + 1$, $t \subseteq s_{n + 1}$ and $m \in [2^{e + n + 1 + 3}, 2^{e + n + 1 + 4})$, define a coloring $c_{e, t, m} \colon \powerset(\N) \to \{0, 1\}$ by
\[
  c_{e, t, m}(B) =
  \begin{cases}
    1 &\text{if } \Phi_e(t \cup B, m)\conv = 1\\
    0 &\text{otherwise}.
  \end{cases}
\]
Next, define a coloring $c$ on $\powerset(\N)$ by setting $c(B)$ to be the sequence
\[
  c(B) = \seq{c_{e, t, m}(B) \mid e < n + 1, t \subseteq s_{n + 1}, m \in [2^{e + n + 1 + 3}, 2^{e + n + 1 + 4})}.
\]
Note that $c$ has finite range and can easily be seen to be Borel. Thus by the corollary to the Galvin-Prikry theorem, we can choose $B_{n + 1}$ to be an infinite subset of $B_n$ such that $c$ assigns the same value to all infinite subsets of $B_{n + 1}$. It is straightforward to check that $B_{n + 1}$ has the properties desired.

Finally, recall that we must forbid some elements from $A$. For each $e < n + 1$ and $t \subseteq s_{n + 1}$, if there is any $m$ in the interval $[2^{e + n + 1 + 3}, 2^{e + n + 1 + 4})$ such that $c_{e, t, m}$ has constant value $1$ on $B_{n + 1}$ then pick the least such $m$ and forbid it from $A$. Also recall that at the end of the construction, we take $A$ to consist of all numbers not forbidden from $A$. 

This concludes the construction of $G$ and $A$. All that remains now is to check that they have the necessary properties.

\medskip\noindent\textbf{The set $\bm{G}$ satisfies all requirements.}
Fix any $e$ and we will show that requirement $e$ is satisfied. Suppose $C$ is an infinite subset of $G$. We will show that either $\Phi(C)$ is not a subset of $A$ or it is not infinite.

For each $n$, define $t_n = C\cap s_n$ and note that since $G \subseteq s_n\cup B_n$, $C$ is compatible with $(t_n, B_n)$. We know that for each $n > e$, one of two possibilities holds: 
\begin{enumerate}
    \item there is some $m \in [2^{e + n + 3}, 2^{e + n + 4})$ such that $m$ is forbidden from $A$ and for all infinite sets $D$ compatible with $(t_n, B_n)$, $\Phi(D, m) \conv = 1$
    \item or for every $m \in [2^{e + n + 3}, 2^{e + n + 4})$ and every infinite set $D$ compatible with $(t_n, B_n)$, $\Phi(D, m) \neq 1$.
\end{enumerate}

On the one hand, if the first possibility holds for any $n$ then there is some $m$ such that $\Phi(C, m)\conv = 1$ and $m$ is forbidden from $A$, which implies that $\Phi(C)$ is not a subset of $A$.

On the other hand, if the second possibility always holds then for every $n > e$ and every $m$ in the interval $[2^{e + n + 3}, 2^{e + n + 4})$, $\Phi(C, m) \neq 1$. Since the union of these intervals consists of all numbers greater than or equal to $2^{2e + 4}$, this shows that either $\Phi(C)$ is not total or it only contains numbers less than $2^{2e + 4}$ and thus is not infinite.

\medskip\noindent\textbf{The set $\bm{A}$ has high density.}
For each $e$, let $A_e$ denote the set of numbers forbidden from $A$ on behalf of requirement $e$. Note that $\bar{A} = \bigcup_e A_e$. We will show that each $A_e$ never has density greater than $1/2^{e + 2}$ and thus that their union never has density more than
\[
  \sum_{e \in \N} \frac{1}{2^{e + 2}} = \frac{1}{2}.
\]
This implies that $\bar{A}$ never has density more than $1/2$ and thus that $A$ is $1/2$-dense. Note, by the way, that when we say that each $A_e$ never has density greater than $1/2^{e + 2}$, we mean that for each $n$, $A_e$ is at most $1/2^{e + 2}$-dense at $n$ (in the sense of Section~\ref{sec:prelim_density}); it is not good enough for it to just have upper density at most $1/2^{e + 2}$.

So fix $e$ and we will argue that $A_e$ never has density more than $1/2^{e + 2}$. First consider a single interval of the form $[2^{e + n + 3}, 2^{e + n + 4})$ for $n > e$. The only time numbers from this interval will be added to $A_e$ is on step $n$ of the construction and on that step, at most one number will be added to $A_e$ per subset of $s_n$. Since $s_n$ has size exactly $n$, this means at most $2^n$ such numbers will be added. Thus we have
\[
  |A_e \cap [2^{e + n + 3}, 2^{e + n + 4})| \leq 2^n = \frac{1}{2^{e + 3}}\cdot |[2^{e + n + 3}, 2^{e + n + 4})|.
\]
In other words, in the interval $[2^{e + n + 3}, 2^{e + n + 4})$, $A_e$ has density at most $1/2^{e + 3}$.

From this fact it can easily be shown that at each power of $2$, $A_e$ has density at most $1/2^{e + 3}$, i.e.\ for each $m > 0$,
\[
  |A_e\cap [2^m - 1]| \leq \frac{2^m}{2^{e + 3}}.
\]
It remains to check that the density of $A_e$ is high enough in-between powers of $2$. But this follows easily from what we have already established. If $2^m \leq k < 2^{m + 1}$ then we have
\begin{align*}
  |A_e \cap [k]| &\leq |A_e\cap [2^{m + 1} - 1]|\\
                 &\leq \frac{2^{m + 1}}{2^{e + 3}} = \frac{2^m}{2^{e + 2}} \leq \frac{k + 1}{2^{e + 2}}
\end{align*}
and thus $A_e$ is at most $1/2^{e + 2}$-dense at $k$, as promised.
\end{proof}

The proof of Theorem~\ref{thm:ds_nle_rt} actually yields a somewhat stronger result. Let $\RT^1_{<\infty}$ denote the problem in which an instance is a finite partition $A_1,\ldots, A_n$ of $\N$ and a solution is an infinite set $B$ which is a subset of some $A_i$. Roughly speaking, $\RT^1_{<\infty}$ is the problem corresponding to Corollary~\ref{cor:seetapun2} of Seetapun's theorem and thus the corollary below shows that there is no direct proof of our main theorem from Corollary~\ref{cor:seetapun2} in the same sense that the theorem above showed there is no direct proof of our main theorem from Seetapun's theorem.

\begin{corollary}
$\DS \nleq_{soc} \RT^1_{< \infty}$.
\end{corollary}

\begin{proof}
Let $A$ and $G$ be as in the proof of Theorem~\ref{thm:ds_nle_rt} and suppose for contradiction that $\DS \leq_{soc} \RT^1_{< \infty}$. Then, thinking of $A$ as an instance of $\DS$, we get a partition $B_1,\ldots,B_n$ of $\N$ such that for any $B_i$ and any infinite subset $C \subseteq B_i$, $C$ computes an infinite subset of $A$. However, since $G$ is infinite, there must be some $B_i$ such that $G\cap B_i$ is infinite. Since $G\cap B_i$ is an infinite subset of $G$, it does not compute any infinite subset of $A$, which gives a contradiction.
\end{proof}

\bibliographystyle{plain}
\bibliography{dense_coding}

\end{document}